\documentclass[letterpaper,11pt]{article}
\usepackage[margin=1in]{geometry}  

\usepackage{amsmath,graphicx,enumitem,bbm,amssymb}
\usepackage{mhequ}
\usepackage{amsthm}
\usepackage[colorlinks=true,citecolor=red,linkcolor=blue,urlcolor=blue]{hyperref}

\theoremstyle{plain}
\newtheorem{theorem}{Theorem}

\newtheorem{lemma}{Lemma}

\newtheorem{corollary}{Corollary}

\numberwithin{theorem}{section}
\numberwithin{corollary}{section}
\numberwithin{lemma}{section}
\numberwithin{equation}{section}

\newcommand{\be}{\begin{equs}}
	\newcommand{\ee}{\end{equs}}

\def \E{\mathbb{E}}
\def \P{\mathbb{P}}

\def \bP {\mathbb{P}}
\def \bE {\mathbb{E}}
\def \bR {\mathbb{R}}

\def \cB {\mathcal{B}}
\def \cN {\mathcal{N}}
\def \var {\mathsf{Var}}

\begin{document}
	
	\title{On Estimation of $L_{\lowercase{r}}$-Norms in Gaussian White Noise Models}
	\author{Yanjun Han, Jiantao Jiao, and Rajarshi Mukherjee\thanks{Yanjun Han is with the Department of Electrical Engineering, Stanford University, 350 Serra Mall, Stanford, CA 94305, USA. Email: \url{yjhan@stanford.edu}. Jiantao Jiao is with the Department of Electrical Engineering and Computer Sciences, University of California, Berkeley, 257M Cory Hall, Berkeley, CA 94720, USA. Email: \url{jiantao@eecs.berkeley.edu}. Rajarshi Mukherjee is with the Department of Biostatistics, Harvard T. H. Chan School of Public Health, Harvard University, 655 Huntington Avenue, Boston, MA 02115, USA. Email: \url{ram521@mail.harvard.edu}. Yanjun Han and Jiantao Jiao were partially supported by the NSF Center for Science of Information under Grant CCF-0939370, and Jiantao Jiao was partially supported by NSF Grant IIS-1901252.}}

	\maketitle
	
	\begin{abstract}
We provide a complete picture of asymptotically minimax estimation of $L_r$-norms (for any $r\ge 1$) of the mean in Gaussian white noise model over Nikolskii-Besov spaces. In this regard, we complement the work of \cite{lepski1999estimation}, who considered the cases of $r=1$ (with poly-logarithmic gap between upper and lower bounds) and $r$ even (with asymptotically sharp upper and lower bounds) over H\"{o}lder spaces. We additionally consider the case of asymptotically adaptive minimax estimation and demonstrate a difference between even and non-even $r$ in terms of an investigator's ability to produce asymptotically adaptive minimax estimators without paying a penalty.
	\end{abstract}	
	\tableofcontents
\section{Introduction}
Estimation of functionals of data generating distributions is a fundamental problem in statistics. Whereas relevant issues in finite dimensional parametric models are comparatively well understood \cite{bickel1993efficient,van2000asymptotic}, corresponding nonparametric analogues are often much more challenging and have attracted tremendous interest over the last two decades. In this regard, initial efforts have focused on the inference of linear and quadratic functionals in Gaussian white noise and density models and have contributed immensely to the foundations of ensuing research. We do not attempt to survey the extensive literature in this area. However, the interested reader can find a comprehensive snapshot of the literature in \cite{hall1987estimation,bickel1988estimating,donoho1990minimax1,donoho1990minimax2,fan1991estimation,birge1995estimation,kerkyacharian1996estimating,laurent1996efficient,nemirovski2000topics,cai2003note,cai2004minimax,cai2005nonquadratic,tchetgen2008minimax} and the references therein.

For treatment of more general smooth functionals in Gaussian White Noise model (for smoothness measured in terms of differentiability in $L_2$), the excellent monograph of \cite{nemirovski2000topics} provides detailed analyses and references of cases where efficient parametric rate of estimation is possible. Further, in  recent times, some progress has also been made towards the understanding of more complex nonparametric functionals over substantially more general observational models. These include causal effect functionals in observational studies and mean functionals in missing data models. For more details, we refer to \cite{robins2008higher,robins2015higher,mukherjee2017semiparametric}, which considers a general recipe to yield minimax estimation of a large class of nonparametric functionals common in statistical literature. However, apart from general theory of estimating linear functionals, most of the research endeavors, at least from the point of view of optimality, have focused on ``smooth functionals" (see \cite{robins2008higher} for a more discussions on general classes of ``smooth functionals"). 

In contrast, the results on the asymptotically minimax estimation of non-smooth  functionals have been comparatively sporadic (\cite{ibragimov1981,Korostelev1991,korostelev2012minimax}). The paradigm got an impetus from the seminal papers of \cite{lepski1999estimation} and \cite{cai2011testing} which considered the estimating of $L_r$-norms in Gaussian mean models. Subsequently, significant progress has been made regarding inference of non-smooth functionals in discrete distribution settings (\cite{valiant2011power,jiao2015minimax,wu2016minimax,han2016minimax,jiao2016minimax,jiao2017maximum}). However, even in the simpler setting of Gaussian white noise model, a complete picture of minimax optimality for estimating integrated non-smooth functionals remain unexplored. This paper is motivated by taking a step in that direction by providing a complete description of asymptotically minimax estimation of $L_r$-norms (for $r\ge 1$) of the mean in Gaussian white noise model over Nikolskii-Besov spaces. We additionally consider the case of adaptive minimax estimation and demonstrate a difference between even and non-even $r$ in terms of an investigator's ability to produce asymptotically adaptive minimax estimators without paying a penalty. 

More specifically, we consider noisy observation $\{Y(t)\}_{t\in [0,1]}$ in the Gaussian white noise model with known variance $\sigma^2$ as
\be 
dY(t)=f(t)dt+\frac{\sigma}{\sqrt{n}}dB(t), \label{eqn:model}
\ee
where $f:[0,1]\rightarrow \mathbb{R}$ is the unknown mean function and  $\{B(t)\}_{t\in [0,1]}$ is the standard Brownian motion on $[0,1]$. The main goal of this paper is to consider adaptive minimax estimation of the $L_r$-norm of the mean function $f$ (i.e. $\|f\|_r\triangleq (\int_{[0,1]} |f(t)|^rdt)^{1/r}$ for $r\ge 1$) over Nikolskii-Besov spaces $\cB_{p,\infty}^s(L)$ in $L_p[0,1]$, $p\ge 1$ of smoothness $s>0$ (defined in Section \ref{section:function_spaces_approx}). It is worth noting here, that the specific cases of the mean function $f$ being uniformly bounded away from $0$ is significantly easier since in that case a natural plug-in principle yields asymptotic optimality. 

As mentioned earlier, significant progress towards understanding these specific functionals has been made in \cite{lepski1999estimation} and \cite{cai2011testing}. In particular, \cite{lepski1999estimation} considers estimation of the $L_r$-norm over H\"{o}lder spaces of known smoothness and demonstrate rate optimal minimax estimation for $r$ even positive integers. For $r=1$ their results are suboptimal and leave a poly-logarithmic gap between the upper and lower bounds for the rate of estimation. Moreover, for general non-even $r$, \cite{lepski1999estimation} provides no particular estimator. Finally, their results are non-adaptive in nature and requires explicit knowledge of the smoothness index of the underlying function class. Our main contribution is improving the lower bound argument (over function spaces similar to H\"{o}lder balls) to close the gap in non-adaptive minimax estimation of $L_r$-norm of the signal function. Moreover, for general non-even $r\ge 1$, the analysis extends further to demonstrate adaptive minimax estimators without logarithmic penalties which are typical in smooth functional estimation problems. However, the situation is different for even integers $r$ where our results show that a poly-logarithmic penalty is necessary. In this effort, the fundamental work of   \cite{cai2011testing}, which  
considered estimating the $L_1$-norm of the mean of an $n$-dimensional multivariate Gaussian vector, serves as a major motivation. In Section \ref{section:main_results} we comment more on the main motivating ideas from \cite{cai2011testing} as well as the fundamental differences and innovations. 

The main results of this paper are summarized below. 
\begin{enumerate}
	\item [(a)] We produce minimax rate optimal estimator $L_r$ norm of the unknown mean function $f$ in Gaussian White noise model \eqref{eqn:model} with known variance. 
	
	\item [(b)]  For non-even $r$, an accompanying adaptive minimax optimal estimator is also provided. In contrast, for even integers $r$, we argue along the lines of standard results from \cite{ingster1987minimax,ingster2012nonparametric} that at least poly-logarithmic penalty needs to be paid for adaptation. The lower bound on this penalty is not sharp in this regard and only serves to demonstrates the lack of adaptation without paying a price. 
	

	\item [(c)] Similar to \cite{cai2011testing}, both our upper and lower bounds rely on best polynomial approximations of suitable functions on the unit interval, which might be of independent interest.
\end{enumerate}

\subsection{Organization} The rest of the paper is organized as follows. In Section \ref{section:function_spaces_approx} we discuss function spaces relevant to this paper as well as best polynomial approximations of continuous functions on compact intervals along with properties of Hermite polynomials, which are useful ingredients in the construction of our estimators. Section \ref{section:main_results} contains the main results of the paper and is divided into two main subsections based on the non-even or even nature of $r$ while estimating $\|f\|_r$. For non-even $r\ge 1$ (Section \ref{section:r_odd}), we first lay down the basic principles for $r=1$ (Section \ref{section:r_equals_one}) followed by the general $r$ (Section \ref{section:r_gtrone_odd}). The case of even $r$ is well understood from \cite{lepski1999estimation} and  is mostly presented here (Section \ref{section:r_even}) for completeness and discussing issues of adaptive estimation. In Section \ref{section:discussions} we discuss remaining issues and future directions. Proofs of the main theorems are collected in Section \ref{section:main_proofs} followed by proofs of several technical lemmas in Section \ref{section:technical_lemmas}.

\subsection{Notation} 
In this paper, $ \mathsf{Poly}_K$ denotes the set of all polynomials over $[-1,1]$ with real coefficients and degree at most $K$. For any finite set $S$ we denote its cardinality by $|S|$. For a function defined on $[0,1]$, for $1\leq q <\infty $ we let $\|h\|_q:=(\int_{[0,1]} |h(x)|^q dx)^{1/q}$ denote the $L_q$ semi-norm of $h$, $\|h\|_{\infty}:=\sup_{x \in [0,1]}|h(x)|$ the $L_{\infty}$ semi-norm of $h$. We say $h \in L_q[0,1]$ for $q\in [1,\infty]$ if $\|h\|_q<\infty$.  The results in this paper are mostly asymptotic (in $n$) in nature and thus requires some standard asymptotic  notations.  If $a_n$ and $b_n$ are two sequences of real numbers then $a_n \gg b_n$ (and $a_n \ll b_n$) implies that ${a_n}/{b_n} \rightarrow \infty$ (and ${a_n}/{b_n} \rightarrow 0$) as $n \rightarrow \infty$, respectively. Similarly $a_n \gtrsim b_n$ (and $a_n \lesssim b_n$) implies that $\liminf{{a_n}/{b_n}} = C$ for some $C \in (0,\infty]$ (and $\limsup{{a_n}/{b_n}} =C$ for some $C \in [0,\infty)$). Alternatively, $a_n=o(b_n)$ will also imply $a_n \ll b_n$ and $a_n=O(b_n)$ will imply that $\limsup{{a_n}/{b_n}} =C$ for some $C \in [0,\infty)$. Finally we comment briefly on the various constants appearing throughout the text and proofs. Given that our primary results concern convergence rates of various estimators, we will not emphasize the role of constants throughout and rely on fairly generic notation for such constants. In particular, for any fixed  tuple $v$ of real numbers, $C(v)$ will denote a positive real number which depends on elements of $v$ only. Finally, whenever we use the symbol $\lesssim$ in the asymptotic sense above, the hidden positive constant $C$ will depend on the known parameters of the problem.

\section{Function Spaces and Approximation}\label{section:function_spaces_approx} 
We begin with some standard definitions of function spaces \cite{Devore--Lorentz1993,hardle2012wavelets} that we work with throughout. In the study of nonparametric functional estimation problem, many studies were devoted to the case where $f$ is assumed to lie in a H\"{o}lder ball defined as
\be
H(s,L) \triangleq \left\{f\in L^2[0,1]: \frac{|f^{(r)}(x)-f^{(r)}(y)|}{|x-y|^{\alpha}}\le L, \forall x\neq y\in [0,1]\right\},
\ee
where $s=r+\alpha>0$ is the smoothness parameter, $r\in\mathbb{N}, \alpha\in(0,1]$. In this paper, we consider another function class which is close but not identical to the H\"{o}lder ball where the dependence of the upper and lower bounds on $n$ matches.
The $r$-th symmetric difference operator $\Delta_h^r$ is defined as \cite{Devore--Lorentz1993}
\be
\Delta_h^r f(x)= \sum_{k=0}^r (-1)^{r-k}\binom{r}{k}f(x+(k-\frac{r}{2})h),
\ee
with the agreement that $\Delta_h^r f(x)=0$ when either $x+\frac{r}{2}h$ or $x-\frac{r}{2}h$ does not belong to $[0,1]$. Then the $r$-th order modulus of smoothness is defined as \cite{Devore--Lorentz1993}
\be \label{eq.moduli_of_smoothness}
\omega^r(f,t)_p = \sup_{0<h\le t} \|\Delta_h^r f\|_p,
\ee
with $p\in [1,\infty]$. Now define the Besov norm of a function $f$ as \cite{besov1979integral}
\be \label{eq.besov_norm}
\|f\|_{\cB_{p,q}^s} = \|f\|_p + \begin{cases}
	\left[\int_0^\infty \left(\frac{\omega^r(f,t)_p}{t^s}\right)^q\cdot \frac{dt}{t}\right]^{\frac{1}{q}} & 1\le q<\infty\\
	\sup_{t>0} \frac{\omega^r(f,t)_p}{t^s} & q=\infty
\end{cases}, 
\ee
with parameters $s>0, p,q\in[1,\infty]$, and $r=\lfloor s\rfloor+1$. Then the corresponding Besov ball is defined by
\be
\cB_{p,q}^s(L) \triangleq \{f\in L^p[0,1]: \|f\|_{\cB_{p,q}^s} \le L\}.
\ee

We note that the definition of Besov ball is a generalization of H\"{o}lder ball via the relationship $H(s,L')=\cB_{\infty,\infty}^s(L)$ for non-integer $s$. Moreover, by the monotonicity of $L_p$ norms on $[0,1]$, we have
$
\cB_{p,q}^s(L) \supseteq \cB_{p',q}^{s}(L)$ for $p\le p'.
$
As a result, for $1\le p<\infty$, $\cB_{p,\infty}^s(L)\supseteq \cB_{\infty,\infty}^s(L)$ is a function class slightly larger than the H\"{o}lder ball $H(s,L')$. In this paper we work with $\cB_{p,\infty}^s(L)$ for $s\ge 0, p\geq 1$. These spaces are related to Nikolskii-spaces (see \cite{nikolski77} for relevant embeddings) and hence we shall refer them to as Nikolskii-Besov spaces throughout.

Polynomial approximations of continuous functions on compact intervals around the origin plays an important role in this paper. To introduce the basic ideas, consider the following best degree-$K$ polynomial approximation of $|\mu|^r$ on $[-1,1]$:
\be
\sum_{k=0}^K g_{K,k}^{(r)}\mu^k \triangleq \arg\min_{Q\in \mathsf{Poly}_K} \max_{\mu\in [-1,1]} |Q(\mu)-|\mu|^r|.
\ee

In order to estimate such polynomials based on a sample $X\sim \cN(\mu,1)$, we will need the notion of Hermite polynomials. In particular, the Hermite polynomial of degree $k$ defined by 
\be\label{eqn:hermite}
H_k(x) \triangleq (-1)^k \exp(\frac{x^2}{2})\cdot \frac{d^k}{dx^k}\left[\exp\left(-\frac{x^2}{2}\right)\right]. 
\ee
The properties of the Hermite polynomials in the context of estimating moments of Gaussian random variables will be important for us and are summarized in Lemma \ref{lemma:hermite}.

\section{Main Results}\label{section:main_results}
We divide our main results in two subsections based on the non-even or even nature of $r$. In particular, the construction of our estimator changes according to this distinction of $r$. However, before we go into the details of these constructions, we need a few definitions.

Consider the kernel projection $f_h(x)$ of $f(x)$ defined as
\be\label{eqn:kernel_projection}
f_h(x) = \int_0^1 f(u)\cdot \frac{1}{h}K_{M}\left(\frac{x-u}{h}\right)du, \quad x\in [0,1], 
\ee
where $K_M(\cdot)$ is a kernel which maps all polynomials of degree at most $M$ to themselves, and $\int_0^1 |K_M(u)|^{M+1}du<\infty$. The choice of $M$ will be clear from the statements of the main results in Section \ref{section:main_results}. We assume that $K_M$ is supported on $[-\frac{1}{2},\frac{1}{2}]$, and the boundary of $f_h(x)$ (i.e., $x\in [0,h/2]$ or $x\in [1-h/2,1]$) is handled using the same way as \cite{lepski1999estimation}.

The corresponding unbiased kernel estimator of $f_h(x)$ defined as
\be
\tilde{f}_h(x) = \int_0^1 \frac{1}{h}K_M\left(\frac{x-u}{h}\right)dY(u)
\ee
admits a usual decomposition into deterministic and stochastic components as follows: 
\be\label{eqn:gauss_model}
\tilde{f}_h(x) = f_h(x) + \lambda_h\xi_h(x).
\ee
Above
\be\label{def:lambda_h}
\lambda_h &= \sqrt{\bE\left(\int_0^1 \frac{1}{h}K_M\left(\frac{x-u}{h}\right)\cdot \frac{\sigma}{\sqrt{n}}dB(u)\right)^2} = \frac{\sigma\|K_M\|_2}{\sqrt{nh}},\\
\xi_h(x) &= \frac{1}{h\lambda_h} \int_0^1 K_M\left(\frac{x-u}{h}\right)\cdot \frac{\sigma}{\sqrt{n}}dB(u).
\ee
Clearly $\xi_h(x)\sim\cN(0,1)$ and random variables $\xi_h(x)$ and $\xi_h(y)$ are independent when $|x-y|>h$.

The reason for introducing the kernel projection estimator is simple and standard in nonparametric statistics. 
In particular, for a suitable chosen bandwidth $h$, $\|f_h-f\|_r$ is small, and it suffices to consider estimation of $\|f_h\|_r$ based on the Gaussian model \eqref{eqn:gauss_model}. Indeed, a crucial part is to estimate $|f_h(x)|^r$. Whereas, for $r$ an even positive integer, this task is relatively simpler \cite{lepski1999estimation}, the case of non-even $r$ poses a more subtle problem due to non-differentiability of the function $u\mapsto |u|^r$ near the origin. Consequently, for such cases, recent techniques for estimating non-smooth functionals needs to employed. 

\subsection{Non-even $r$}\label{section:r_odd} In this case, the construction depends on the best polynomial approximation of the function $u\mapsto |u|^r$ over the interval $[-1,1]$ and borrows heavily from a recent line of work by \cite{cai2011testing,jiao2015minimax,han2016minimax,wu2016minimax}. The general principle of the construction goes along the following heuristic steps.
\begin{itemize}
	\item Approximate $f$ by a kernel projection $f_h$ (as in \eqref{eqn:kernel_projection}) and consider estimating $\|f_h\|_r$ instead at a cost of incurring a truncation bias.
	
	\item  A ``large value" of the kernel estimator $|\tilde{f}_h(x)|$ (referred to as ``smooth regime" hereafter) gives indication of a corresponding ``large value" of $|f_h(x)|$ and a plug-in type estimator for $|f_h(x)|^r$ is reasonable.
	
	\item A relatively ``small value" of $|\tilde{f}_h(x)|$ (referred to as ``non-smooth regime" hereafter) gives indication of a correspondingly ``small value" of $|f_h(x)|$ and a plug-in type estimator for $|f_h(x)|^r$ is no-longer reasonable owing to the non-differentiable nature of the absolute function near the origin. In this case, similar to \cite{cai2011testing,jiao2015minimax,wu2016minimax,han2016minimax}, an estimator based on the best polynomial approximation of the function $u\mapsto |u|^r$ is employed. 
	
	\item The final estimator integrates over this two regimes of $|\tilde{f}_h(x)|$ followed by an optimal choice of $h$ to trade off squared bias and variance.
\end{itemize}
Below we make the program laid down above more concrete and refer readers to \cite{jiao2015minimax,han2016minimax} for a detailed discussion of the general principle of estimating non-smooth functionals. The same procedure also works for estimating other non-smooth nonparametric functionals, e.g., the differential entropy \cite{han-jiao-weissman-wu2017minimax}. It turns out that the treatment for $r=1$ is easier, more transparent and slightly different than general non-even $r>1$. Consequently, we will present the case for $r=1$ first for the sake of clarity followed by the more general case. 

Any candidate estimator below will be defined by the parameter tuple $(h,c_1,c_2,\epsilon)$ where $h$ is the bandwidth of the kernel projection \eqref{eqn:kernel_projection} and $(c_1,c_2,\epsilon)$ are constants depending on the known parameters of the problem (i.e. the variance $\sigma^2$ and radius of Besov balls $L$) to be specified later. More specifically, $c_1$ will be chosen as large as possible whereas $c_2$ and $\epsilon$ will be desirably small. In general we suppress the dependence of our estimators on $(c_1,c_2,\epsilon)$. However, our adaptive estimator makes a data driven choice of the bandwidth $h$. Therefore we index our estimator by this bandwidth, namely, $T_h$.

\subsubsection{$r=1$} \label{section:r_equals_one}
We follow the general principle laid down above. 
Recall that $\{g_{K,k}^{(1)}\}_{k=0}^K$ are the coefficients of the best polynomial approximation of $u\mapsto |u|$ of degree $K$ on $[-1,1]$ and $H_k$ is the Hermite polynomial of degree $k$. With this in mind, the construction of our estimator $T_h$ for every bandwidth $h$ can be described the following steps. 
\begin{enumerate}
	\item [(I)]  Using the sample splitting technique for the Brownian motion \cite{nemirovski2000topics} to obtain two independent observations $\tilde{f}_{h,1}(x)$ and $\tilde{f}_{h,2}(x)$ for any $x\in [0,1]$. This reduces the effective sample size $n$ by half and for simplicity of notation we redefine $n/2$ as $n$.

	\item[(II)] For any $x$ define an estimator of $|f_h(x)|$ as
	\be
	T_h(x) = |\tilde{f}_{h,1}(x)| \mathbbm{1}(|\tilde{f}_{h,2}(x)|\ge c_1\lambda_h\sqrt{\ln n}) + \tilde{P}(\tilde{f}_{h,1}(x)) \mathbbm{1}(|\tilde{f}_{h,2}(x)|< c_1\lambda_h\sqrt{\ln n}),
	\ee
	where \be
	\tilde{P}(u) &= \min\{\max\{P(u),-n^{2\epsilon}\lambda_h\}, n^{2\epsilon}\lambda_h\}, \quad \text{with}\\
	P(u) &= \sum_{k=0}^K g_{K,k}^{(1)} \left(2c_1\lambda_h\sqrt{\ln n}\right)^{1-k}\cdot\lambda_h^kH_k\left(\frac{u}{\lambda_h}\right),
	\ee
	where $K=\lceil c_2\ln n \rceil$.

	\item [(III)] Finally, an estimator for $\|f\|_1$ is defined as  
	\be
	T_h \triangleq \min\left\{L, \max\left\{0,\int_0^1 T_h(x)dx\right\}\right\}.
	\ee
\end{enumerate}

The following Theorem describes the mean squared error of estimating $\|f\|_1$ by $T_h$ over $\cB_{p,\infty}^s(L)$. 

\begin{theorem}\label{theorem:lone_bias_var}
	Choose $M>\lceil s\rceil$ and consider the corresponding kernel projection $f_h$ based on $K_M$ defined by \ref{eqn:kernel_projection}. Suppose $(h,c_1,c_2,\epsilon)$ satisfy $4c_1^2\ge c_2, c_2\ln n\ge 1, c_1>8, 7c_2\ln 2< \epsilon\in (0,1)$. Then for any $p\in [1,\infty]$, we have
	\be 
	\left(\sup_{f\in \cB_{p,\infty}^s(L)}\E_f\left(T_h-\|f\|_1\right)^2\right)^{\frac{1}{2}}&\leq C \left(h^s+\frac{1}{\sqrt{nh\ln{n}}}+\frac{1}{n^{(1-\epsilon)/2}}\right),
	\ee
	for a constant $C>0$ depending on $s,p,L,\sigma$.
\end{theorem}
As an immediate consequence of Theorem \ref{theorem:lone_bias_var}, by choosing $\epsilon>0$ sufficiently small (i.e., choose a small $c_2$) and by setting
\be
h\asymp (n\ln n)^{-\frac{1}{2s+1}}, 
\ee
we have the following result. 

\begin{corollary}\label{corollary:lone_upper_bound}
	Under the assumptions of Theorem \ref{theorem:lone_bias_var}, for $h\asymp (n\ln n)^{-\frac{1}{2s+1}}$ we have
	\be
	\left(\sup_{f\in\cB_{p,\infty}^s(L)}\bE_f\left(T_h-\|f\|_1\right)^2\right)^{\frac{1}{2}}\leq  C(n\ln n)^{-\frac{s}{2s+1}},
	\ee
	for a constant $C>0$ depending on $s,p,L,\sigma$.
\end{corollary}
The same asymptotic upper bound was demonstrated with a different estimator by \cite{lepski1999estimation}. However, their results were over H\"{o}lder Balls corresponding to $p=\infty$ and require a uniform upper bound on $\|f\|_{\infty}$ for all $f$. In contrast, we do not require any such knowledge of upper bound and produce results for any $p\in [1,\infty]$. Moreover, the next theorem shows that our results are rate optimal in terms of matching lower bounds for any $1\leq p<\infty$. 
\begin{theorem}\label{theorem:lone_lower_bound}
	For any $1\leq p<\infty$, we have
	\be
	\left(\inf_{\hat{T}}\sup_{f\in\cB_{p,\infty}^s(L)}\bE_f\left(\hat{T}-\|f\|_1\right)^2\right)^{\frac{1}{2}}\geq C' (n\ln n)^{-\frac{s}{2s+1}},
	\ee
	for a constant $C'>0$ depending on $s,p,L,\sigma$ and where the infimum above is taken over all measurable maps of $\{Y(t)\}_{t\in[0,1]}$.
\end{theorem}
Corollary \ref{corollary:lone_upper_bound} along with Theorem \ref{theorem:lone_lower_bound} provide a complete picture of the minimax rate of estimation of $\|f\|_1 $ over $\cB_{p,\infty}^s(L)$ for any $1\leq p<\infty$. We remark that our lower bound proof does not provide matching results for $p=\infty$, and thus the gap in the exact framework considered by \cite{lepski1999estimation} remains. However, our result provides strong moral evidence that upper bound of Theorem 2.1 of \cite{lepski1999estimation} is rate optimal and it is the lower bound that stands to be improved.

\subsubsection{$r>1$, non-even}\label{section:r_gtrone_odd}

For $r>1$ not an even integer, the general philosophy of the construction of a candidate estimator $T_h$ is similar to the case of $r=1$ in Section \ref{section:r_equals_one} and follows in three steps as before. However, the simple plug-in principle employed in the smooth regime (i.e., when $|\tilde{f_h}(x)|$ is large) incurs significant bias. As a consequence, Step II of the construction needs to be modified based on the following heuristics of Taylor expansion. The reason why we require $r$ to be non-even is that the function $u\mapsto |u|^r$ is not an analytical function in this case.

Using the notation from Section \ref{section:r_equals_one}, the heuristic Taylor expansion gives
\be\label{eqn:Taylor}
f_h(x)^r \approx \sum_{k=0}^{R} \frac{r(r-1)\cdots(r-k+1)}{k!}(\tilde{f}_{h,1}(x))^{r-k}(f_h(x)-\tilde{f}_{h,1}(x))^k,
\ee
where we choose $R=\lfloor 2r\rfloor$ to reduce the approximation error of \eqref{eqn:Taylor} to a desired order (cf. Lemma \ref{lemma:large_regime_lr}). Based on this approximate Taylor expansion, the right hand side of \eqref{eqn:Taylor} is a natural candidate estimator for $f_h(x)^r$. However, such an estimator is infeasible due to its dependence on unknown $f_h(x)$. Consequently, we replace $(f_h(x)-\tilde{f}_{h,1}(x))^k$ in \eqref{eqn:Taylor} by an unbiased estimator constructed as follows: let $\tilde{f}_{h,2}(x)$ be an independent copy of $\tilde{f}_{h,1}(x)$ obtained via the sample splitting technique of the Brownion motion \cite{nemirovski2000topics}. Then Lemma \ref{lemma:hermite} gives
\begin{align*}
\bE_{\tilde{f}_{h,2}}\left[ \sum_{j=0}^k \binom{k}{j}\lambda_h^jH_j\left(\frac{\tilde{f}_{h,2}(x)}{\lambda_h}\right)(-\tilde{f}_{h,1}(x))^{k-j} \right] &= \sum_{j=0}^k \binom{k}{j} f_h(x)^j (-\tilde{f}_{h,1}(x))^{k-j} \\
&= (f_h(x) - \tilde{f}_{h,1}(x))^k, 
\end{align*} 
i.e., the estimator inside the expectation is an unbiased estimator of $(f_h(x) - \tilde{f}_{h,1}(x))^k$. 

With the above intuition and notation in mind, the construction of our estimator $T_h$ for every bandwidth $h$ can be described the following steps. 
\begin{enumerate}
	\item [(I)]  Using the sample splitting technique for the Brownian motion \cite{nemirovski2000topics} to obtain three independent observations $\tilde{f}_{h,1}(x), \tilde{f}_{h,2}(x)$ and $\tilde{f}_{h,3}(x)$ for any $x\in [0,1]$. Redefine $n/3$ as $n$ for simplicity.

	\item[(II)] For any $x\in [0,1]$ define an estimator of $|f_h(x)|^r$ as
	\be
	T_h(x) &= \tilde{P}_r(\tilde{f}_{h,1}(x))\mathbbm{1}(|\tilde{f}_{h,3}(x)|\le c_1\lambda_h\sqrt{\ln n}) \\
	&\qquad + S_{\lambda_h}(\tilde{f}_{h,1}(x), \tilde{f}_{h,2}(x))\mathbbm{1}(\tilde{f}_{h,3}(x)> c_1\lambda_h\sqrt{\ln n}) \\
	&\qquad + S_{\lambda_h}(-\tilde{f}_{h,1}(x), -\tilde{f}_{h,2}(x))\mathbbm{1}(\tilde{f}_{h,3}(x)<- c_1\lambda_h\sqrt{\ln n}), 
	\ee
	where 
	\be 
	S_{\lambda_h}(u,v) &= \sum_{k=0}^{R} \frac{r(r-1)\cdots(r-k+1)}{k!}u^{r-k}\cdot\left(\sum_{j=0}^k \binom{k}{j}\lambda_h^jH_j\left(\frac{v}{\lambda_h}\right)(-u)^{k-j}\right) \\
	&\qquad \cdot \mathbbm{1}\left(u\ge \frac{c_1}{4}\lambda_h\sqrt{\ln n}\right), \label{eq.S_lambda}
	\ee
	and
	\be
	\tilde{P}_r(t) &= \min\{\max\{P_r(t),-\lambda_h^rn^{2\epsilon}\}, \lambda_h^rn^{2\epsilon} \},\quad\text{with}\\
	P_r(t) &= \sum_{k=0}^K g_{K,k}^{(r)}\left(2c_1\lambda_h\sqrt{\ln n}\right)^{r-k}\cdot \lambda_h^kH_k(\frac{t}{\lambda_h}),
	\ee
	where $K=\lceil c_2\ln n\rceil$, $H_k$ is the Hermite polynomial of degree $k$ and $\{g_{K,k}^{(r)}\}_{k=0}^K$ is the coefficient of the best polynomial approximation of $u\rightarrow |u|^r$ on $[-1,1]$.
	
	\item [(III)] Finally, the overall estimator for $\|f\|_r$ is 
	\be
	T_h \triangleq \min\left\{L, \left(\max\left\{0,\int_0^1 T_h(x)dx\right\}\right)^{\frac{1}{r}}\right\}.
	\ee
	
\end{enumerate}

The following Theorem describes the optimal mean squared error of estimating $\|f\|_r$ by $T_h$ over $\cB_{p,\infty}^s(L)$. 

\begin{theorem}\label{theorem:lr_odd_bias_var}
	Let $r>1$ be non-even and $p\in [r,\infty]$. Choose $M>\lceil s\rceil$ and consider the corresponding kernel projection $f_h$ based on $K_M$ defined by \eqref{eqn:kernel_projection}. Suppose that $(h,c_1,c_2,\epsilon)$ satisfy $c_1^2\ge 16$, $c_2\ln n\ge 1, 4c_1^2\ge c_2$, $7c_2\ln 2<\epsilon\in(0,1)$ and $h=(n\ln{n})^{-\frac{1}{2s+1}}$. Then 
	\be 
	\left(\sup_{f\in \cB_{p,\infty}^s(L)}\E_f\left(T_h-\|f\|_r\right)^2\right)^{\frac{1}{2}}&\leq C(n\ln n)^{-\frac{s}{2s+1}},
	\ee
	for a constant $C>0$ depending on $s,p,r,L,\sigma$.
\end{theorem}

The next theorem shows that the upper bounds in Theorem \ref{theorem:lr_odd_bias_var} are rate optimal in terms of matching lower bounds for any $1\leq p<\infty$. 
\begin{theorem}\label{theorem:lr_odd_lower_bound}
	For any non-even $r>1$ and $r\leq p<\infty$,
	\be
	\left(\inf_{\hat{T}}\sup_{f\in\cB_{p,\infty}^s(L)}\bE_f\left(\hat{T}-\|f\|_r\right)^2\right)^{\frac{1}{2}}\geq C' (n\ln n)^{-\frac{s}{2s+1}},
	\ee
	for a constant $C'>0$ depending on $s,p,r,L,\sigma$ and where the infimum above is taken over all measurable maps of $\{Y(t)\}_{t\in[0,1]}$.
\end{theorem}
Theorem \ref{theorem:lr_odd_bias_var} along with Theorem \ref{theorem:lr_odd_lower_bound} provides a complete picture of the minimax rate of estimation of $\|f\|_r $ over $\cB_{p,\infty}^s(L)$ for any non-even $r>1$ and $r\leq p<\infty$. This is a generalization of the result in \cite{lepski1999estimation}. 

\subsubsection{Adaptive Estimation}
It is worth noting that the choice of $h$ in Corollary \ref{corollary:lone_upper_bound} and Theorem \ref{theorem:lr_odd_bias_var} depends explicitly on the smoothness index $s$. Consequently, the resulting rate of estimation by $T_h$ is non-adaptive over different possibilities of smoothness. However, the experienced reader will notice the structure of errors in Theorem \ref{theorem:lone_bias_var} for a general $T_h$ indicates a possible data driven adaptive choice of bandwidth $h$. In particular, a Lepski type argument \cite{lepskii1991problem,lepski1992problems,lepskii1992asymptotically,lepskii1993asymptotically} is standard in such situations and turns out to be sufficient for our purpose when $r=1$. The similar construction for general non-even integer $r>1$ is more subtle due to the unavailability of a transparent error analysis as in Theorem \ref{theorem:lone_bias_var}.  Consequently, we describe the adaptive procedure slightly more generally below--without specifically alluding to the case of $r=1$.

Let $r\ge 1$ be non-even, and $s\in (0,s_{\max}]$ for some given $0<s_{\max}<\infty$ and consider adaptive estimation of $\|f\|_r$ over $\cup_{s\in(0,s_{\max}]}\cB_{p,\infty}^s(L)$ with a known $L$. The knowledge of $s_{\max}$ will be necessary for construction of kernels providing optimal approximations in Nikolskii-Besov spaces. Given access to an upper bound $L_{\max}$, our construction of the adaptive estimator can also adapt to the scaling parameter $L$, by noting that the construction of $T_h$ does not require the knowledge of $L$ (except for the final truncation step, where $L$ can be replaced by $L_{\max}$ without affecting the multiplicative constants in Theorems \ref{theorem:lone_bias_var} and \ref{theorem:lr_odd_bias_var}). However, since the dependence of the minimax risk on $L$ is not the main focus of this paper, we assume that $L$ is known and do not elaborate on the adaptation to $L$. In this framework, a Lepski type construction for an adaptive estimator can be achieved as follows. 
\begin{enumerate}
	\item [(I)] Let \be 
	\mathcal{H}&=[h_{\min},h_{\max}]\cap \{1,\frac{1}{2},\frac{1}{3},\cdots \}, \quad \text{where}\\ h_{\min}&=(n\ln{n})^{-1}, \ h_{\max}=(n\ln{n})^{-\frac{1}{2s_{\max}+1}}. 
	\ee 
	
	\item [(II)]
	Define the data-driven bandwidth
	\be 
	\hat{h} \triangleq\max\left\{h\in \mathcal{H}:(T_h-T_{h'})^2 \leq  C^*\frac{\lambda_{h'}^2}{\ln n},\ \forall h'\in  \mathcal{H}, h'\leq h  \right\}. \label{eq.adaptation}
	\ee 
	
	\item[(III)] The final estimator is $\hat{T} = T_{\hat{h}}$. 
\end{enumerate}
Our next theorem justifies the adaptive nature of $\hat{T}$. 

\begin{theorem}\label{theorem:lr_odd_adaptive}
	Let $r\ge 1$ be non-even and $p\in [r,\infty]$. Choose $M>\lceil s_{\max}\rceil$ and consider the corresponding kernel projection $f_h$ based on $K_M$ defined by \eqref{eqn:kernel_projection}. 
	If $4r\epsilon(2s_{\max}+1)<1$ and the constant $C^*$ is large enough (depending on $(p,r,s_{\max},c_1,c_2,\sigma,L,K_M)$),
	\be 
	\left(\sup_{f\in  \cB_{p,\infty}^s(L)}\E_f\left(\hat{T}-\|f\|_r\right)^2\right)^{\frac{1}{2}}&\leq C (n\ln{n})^{-\frac{s}{2s+1}},
	\ee
	for a constant $C>0$ depending on $s,s_{\max},p,r,L,\sigma,c_1,\varepsilon,C^*$.
\end{theorem}


Theorem \ref{theorem:lr_odd_adaptive} shows the existence of an adaptive minimax estimator for $\|f\|_r$ without any penalty on the minimax rate. One of the main challenges in the proof of Theorem \ref{theorem:lr_odd_adaptive} is to demonstrate desired tail bounds of $\{T_h\}_{h \in \mathcal{H}}$, for the candidate estimators in $\{T_h\}_{h \in \mathcal{H}}$ rely on truncated Hermite polynomials of high degrees evaluated at Gaussian random variables.

\subsection{Even $r$}\label{section:r_even} The case of non-adaptive minimax estimation of $\|f\|_r$ for $r$ an even positive integer can be obtained by methods described in \cite{lepski1999estimation}. Although their results were obtained over the H\"{o}lder Balls,
the case of Nikolskii-Besov type spaces that we consider are very similar due to the same nature of approximation error of $f$ by $f_h$. However, for the sake of exposition and completeness, we provide the details here again. 

The crux of the construction on the fact that for $r$ an even positive integer, the function $u\mapsto |u|^r=u^r$ is analytic. Consequently, it is possible to construct unbiased estimator of $\mu^r$ based on samples from $X\sim \cN(\mu,\sigma^2)$. In particular, if $X\sim \cN(\mu,\sigma^2)$ then arguing as Lemma 4.4 of \cite{lepski1999estimation}, $\E\left((X+i\sigma \xi)^r|X\right)$, with $\xi \sim \cN(0,1)$ independent of $X$, is an unbiased estimator of $\mu^r$. As a result, a sequence estimator for $\|f\|_r$ indexed by bandwidth $h$ can now be constructed in the following steps. 

\begin{enumerate}
	\item [(I)] Approximate $f$ by a kernel projection $f_h(x)$ (as in \eqref{eqn:kernel_projection}) and consider estimating $\|f_h\|_r$ instead at a cost of incurring a truncation bias.
	
	\item[(II)] Fix $\xi \sim \cN(0,1)$ independent of $B(t)$. For every $x\in [0,1]$ estimate $f_h(x)^r$ by $T_h(x)\triangleq\E_{\xi}[(\tilde{f}_h(x)+i\lambda\xi)^r]$. 
	
	\item [(III)] Estimate $\|f_h\|^r=\int_{0}^1f_h^r(x)dx$ by $\int_{0}^1 T_h(x) dx$.
	
	\item [(IV)] Finally, an estimator for $\|f\|_r$ is defined as  
	\be
	T_h\triangleq \left(\max\left\{0,\int_{0}^1 T_h(x) dx\right\}\right)^{1/r}.
	\ee
\end{enumerate}

Note that the construction changes from the construction of estimators when $r$ is non-zero only in the definition of $T_h(x)$ and indeed this is due to the ability to produce estimate of analytic functions of Gaussian means. 
The following theorem describes the optimal mean squared error of estimating $\|f\|_r$ by $T_h$ over $\cB_{p,\infty}^s(L)$. 

\begin{theorem}\label{theorem:lr_even_bias_var}
	Let $2\leq r\leq p$ and $r$ be even. Choose $M>\lceil s\rceil$ and consider the corresponding kernel projection $f_h$ based on $K_M$ in \eqref{eqn:kernel_projection}. Letting $h=n^{-1/(2s+1-1/r)}$, 
	\be 
	\left(\sup_{f\in \cB_{p,\infty}^s(L)}\E_f\left(T_h-\|f\|_r\right)^2\right)^{\frac{1}{2}}&\leq Cn^{-\frac{s}{2s+1-1/r}},
	\ee
	for a constant $C>0$ depending on $s,p,r,L,\sigma$.
\end{theorem}
Recall from Section \ref{section:function_spaces_approx} that the approximation error $\|f-f_h\|_r$ is always bounded by $C(L,K_M)h^s$ for any $f\in \cB_{p,\infty}^s(L)$. Hence, the proof of Theorem \ref{theorem:lr_even_bias_var} can be obtained verbatim from the proof of Theorem 2.3 of \cite{lepski1999estimation} and is hence omitted. In fact, proof of the lower bound in Theorem 2.3 of \cite{lepski1999estimation} implies that the proposed $T_h$ with $h=n^{-1/(2s+1-1/r)}$ is in fact asymptotically minimax rate optimal over any $\cB_{p,\infty}^{s}(L)$ as well. Therefore, as before, it remains to explore adaptive minimax estimation over a collection of smoothness classes. In this regard, once can argue along the lines of standard results from \cite{ingster1987minimax,ingster2012nonparametric}, that unlike non-even $r$, adaptation over a range of smoothness indices is not possible without paying a poly-logarithmic penalty. In particular,  consider testing $H_0: f\equiv 0$ vs $H_1: f\in \cB_{r,\infty}^s(L), \|f\|_r\geq \rho_n$ with $s$ varying over a range of values $[s_{\min},s_{\max}]\subseteq (0,\infty)$. Whereas the minimax testing rate of separation for $\rho_n$ with known $s$ equals $n^{-s/(2s+1-1/r)}$ (See \cite{ingster2012nonparametric} and proof of \cite[Theorem 3.4 (b)]{carpentier2013honest} for details), i.e., the minimax rate of estimation of $\|f\|_r$, the adaptation over $[s_{\min},s_{\max}]$ needs an additional penalty for $\rho_n$ which equals $\left(\ln\ln{n}\right)^{C(r,s)}$ for a constant $C(r,s)>0$ depending on $r$ and $s$. The proof of this additional poly-logarithmic penalty is proved for $r=2$ in \cite{spokoiny98adaptive} (see also proof of \cite[Theorem 8.1.1]{gine2015mathematical}). The proof of this additional penalty needed for adaptive hypothesis testing builds on the usual second moment method type lower bound argument for non-adaptive testing and involves putting an additional uniform prior on a suitably discretized subset $[s_{\min},s_{\max}]$. Using similar ideas, the proof of a  required penalty for adaptation for general even $r$-can be obtained by combining proof technique of \cite[Theorem 3.4 (b)]{carpentier2013honest} for non-adaptive testing and adaptive lower bound arguments as in \cite{spokoiny98adaptive,gine2015mathematical}.  The details are omitted and we simply comment on the implications of the result. Indeed, such a poly-logarithmic penalty needed for adaptive hypothesis testing in $L_r$-norm, yields a penalty for adaptive estimation of $\|f\|_r$ over  $\cB_{r,\infty}^s(L)$. We believe that this poly-logarithmic penalty is not sharp for adaptive estimation of $\|f\|_r$ norms for even $r$ and only serves to demonstrates the lack of adaptation without paying a price. In particular, we have shown elsewhere that a careful application of Lepski's method along with an involved computation of all central moments of $T_h$ can yield a penalty which behaves $(\ln{n})^{C'(r,s)}$ for a constant $C'(r,s)>0$ depending on $r$ and $s$. In future work, we plan to explore the exact nature of this poly-logarithmic penalty necessary for adaptation. Finally,  careful readers will notice that, although the poly-logarithmic penalty on adaptive hypothesis testing is ubiquitous for any $r$, the result is interesting to us only for even values of $r$, since otherwise the minimax rate of estimation of $\|f\|_r$ is strictly slower than minimax rate of testing in $L_r$-norm. 

\section{Discussion and Open Problems}\label{section:discussions} In this paper we complement the results in \cite{lepski1999estimation} to provide a complete picture of asymptotically minimax estimation of $L_r$-norm of the mean in a Gaussian White Noise model. Unlike \cite{lepski1999estimation}, our results are rate optimal from both perspectives of upper and lower bounds. In this effort, best polynomial approximation of non-smooth functions plays a major role and might be of independent interest. 

Several interesting questions remain open as challenging future directions. In particular, closing the lower bound gap over H\"{o}lder spaces, the actual premise considered by \cite{lepski1999estimation}, is definitely a question of interest. Understanding the sharp nature of the penalty required for adaptive estimation of $\|f\|_r$ when $r$ is even is another question that remain unanswered. We plan to study these issues in detail in future work.  
\section{Proof of Main Results}\label{section:main_proofs}
\subsection{Proof of Theorem \ref{theorem:lone_bias_var}}
The proof of the Theorem hinges on the following lemma, the derivation of which can be found in Section \ref{section:technical_lemmas}. To state the lemma, consider 
\be
\xi(X,Y) \triangleq \tilde{P}(X)\mathbbm{1}(|Y|< c_1\lambda_h\sqrt{\ln n}) + |X|\mathbbm{1}(|Y|\ge c_1\lambda_h\sqrt{\ln n})
\ee
used for estimating $\mu$ with independent $X,Y\sim\cN(\mu,\lambda_h^2)$ and $\lambda_h$ defined by \eqref{def:lambda_h}. 

\begin{lemma}\label{lemma:lone_xi_bias_var}
	Under the assumptions of Theorem \ref{theorem:lone_bias_var}, for any fixed $\mu$, $k\geq 2$, and $c_1>\sqrt{8k}$, we have that there exists constants $C_1$ (depending on $(c_1,c_2,\epsilon,K_M)$) and $C_2$  (depending on $(c_1,c_2,\epsilon,K_M,k)$) such that
	\be
	|\bE\xi(X,Y)-|\mu||&\leq \frac{C_1}{\sqrt{nh\ln n}}, \\
	\E|\xi(X,Y)-\E(\xi(X,Y))|^k&\leq C_2 \left(\lambda_h n^{2\epsilon}\right)^k.
	\ee
\end{lemma}
Note that Lemma \ref{lemma:lone_xi_bias_var} yields bounds on all even central moments of $\xi(X,Y)$, a result that will be helpful in subsequent proof of adaptation. For now, we will only use the result for $k=2$ which corresponds to bound on the variance of $\xi(X,Y)$.

Coming back to the proof of Theorem \ref{theorem:lone_bias_var}, note that there are three types of errors: 
\begin{enumerate}
	\item Approximation error incurred by replacing $\|f\|_1$ with $\|f_h\|_1$; 
	\item The bias of $\int_0^1 T_h(x)dx$ in estimating $\|f_h\|_1$; 
	\item The variance of $\int_0^1 T_h(x)dx$ in estimating $\|f_h\|_1$.
\end{enumerate}

We bound these errors separately. For the approximation error, by an alternative characterization of Besov balls \cite{hardle2012wavelets}, for $f\in \cB_{p,\infty}^s(L)$ we have that for a constant $C_0$ depending on $L$ and $K_M$, 
\be
\|f-f_h\|_p \leq C_0 h^s, \qquad \forall h>0.
\ee
As a result, the approximation error is upper bounded by
\be
|\|f\|_1 - \|f_h\|_1| \le \|f_h-f\|_1 \le \|f_h-f\|_p \le C_0h^s.
\ee

Secondly we upper bound the bias. By Lemma \ref{lemma:lone_xi_bias_var}, there exists a constant $C_1$ (depending on $(c_1,c_2,\epsilon,K_M)$) such that
\be
\left| \bE \int_0^1 T_h(x)dx - \|f_h\|_1\right| &\le \int_0^1 | \bE T_h(x) - f_h(x)|dx \\
&\leq   \int_0^1 \frac{C_1}{\sqrt{nh\ln n}}dx = \frac{C_1}{\sqrt{nh\ln n}}.
\ee

Finally we upper bound the variance of $T_h$. Note that $T_h(x)$ and $T_h(y)$ are independent as long as $|x-y|>h$. Therefore by Lemma \ref{lemma:lone_xi_bias_var}, there exists a constant $C_2$ (depending on $(c_1,c_2,\epsilon,K_M)$) such that
\be
\var\left(\int_0^1 T_h(x)dx\right)
&= \int_0^1\int_0^1 \mathsf{Cov}(T_h(x),T_h(y)) dxdy\\
&= \iint_{|x-y|\le h} \mathsf{Cov}(T_h(x),T_h(y)) dxdy\\
&\le \iint_{|x-y|\le h} \frac{\var(T_h(x))+\var(T_h(y))}{2} dxdy\\
&=2 h\int_0^1 \var(T_h(x))dx\\
&\leq C_2h\int_0^1 \frac{1}{n^{1-\epsilon}h}dx\\
&= \frac{C_2}{n^{1-\epsilon}}.
\ee

Note that $0\le \|f\|_1\le L$ for any $f\in \cB_{p,\infty}^s(L)$, and consequently
\begin{align*}
\left| T_h - \|f\|_1 \right| \le \left| \int_0^1 T_h(x)dx - \|f\|_1 \right|. 
\end{align*}
In summary, for any $f\in \cB_{p,\infty}^s(L)$, by triangle inequality we have
\be
\left(\bE\left(T_h-\|f\|_1\right)^2\right)^{\frac{1}{2}} &\le \left(\bE\left(\int_0^1 T_h(x)dx-\|f\|_1\right)^2\right)^{\frac{1}{2}}\\
&\leq \sqrt{3}\left(|\|f\|_1-\|f_h\|_1| + \left|\bE \int_0^1 T_h(x)dx-\|f_h\|_1\right| + \sqrt{\var\left(\int_0^1 T_h(x)dx\right)}\right) \\
&\leq C^*\left( h^s+ \frac{1}{\sqrt{nh\ln n}} + \frac{1}{n^{(1-\epsilon)/2}}\right),
\ee
where $C^*$ is a constant depending on $(c_1,c_2,\sigma, \epsilon,L,K_M)$ which in turn satisfies the conditions of Theorem \ref{theorem:lone_bias_var}. This completes the proof of the theorem. \qed

\subsection{Proof of Theorem \ref{theorem:lr_odd_bias_var}}
The proof of the Theorem hinges on the following lemma.

\begin{lemma}\label{lemma:overall_lr_odd}
	Under the conditions of Theorem \ref{theorem:lr_odd_bias_var}, the following hold for all $x\in [0,1]$, $k\geq 2$, $c_1>\sqrt{8k}$, and constants $C_1$ (depending on $(c_1,c_2,\epsilon,\sigma,K_M)$) and $C_2$  (depending on $(c_1,c_2,\epsilon,\sigma,K_M,k)$).
	\be
	|\bE T_h(x) - |f_h(x)|^r| &\leq C_1 (nh\ln n)^{-\frac{r}{2}},\\
	\bE|T_h(x)-\bE T_h(x)|^k &\leq C_2 \frac{n^{2k\epsilon}}{(nh)^{\frac{k}{2}}}\left(|f_h(x)|^{(r-1)k}+(nh)^{-\frac{(r-1)k}{2}}\right).
	\ee
\end{lemma}

We postpone the proof of the lemma to Section \ref{section:technical_lemmas} and complete the proof of Theorem \ref{theorem:lr_odd_bias_var} assuming its validity. 

As is in the case of $L_1$ norm estimation, there are three types of error incurred by our estimator $T_h$, i.e., the approximation error, the bias and the variance. We analyze these errors separately.

For the approximation error, by the property of Besov spaces \cite{hardle2012wavelets} we know that there exists a constant $C_0$ depending on $L$ and $K_M$ such that for all $f\in \cB_{p,\infty}^s(L)$ the kernel approximation error satisfies $\|f-f_h\|_p\leq C_0 h^s$. Hence, by the monotonicity of $L_p$ norms on $[0,1]$ and $r\le p$, we know that
\be
|\|f\|_r - \|f_h\|_r| \le \|f-f_h\|_r \le \|f-f_h\|_p \le C_0h^s.\label{eqn:besov_approx}
\ee
where $C_0>0$ is some universal constant which only depends on radius $L$ and the kernel $K_M$.

For the bias and the variance, we look at the bias and variance of $\Phi_h=\int_{0}^1 T_h(x)dx$,
which is the estimator for $\|f_h\|_r^r$. Similar to the proof of Theorem \ref{theorem:lone_bias_var}, one can show  using Lemma \ref{lemma:overall_lr_odd} that for $C_1$ (depending on $(c_1,c_2,\epsilon,\sigma,K_M)$) and $C_2$  (depending on $(c_1,c_2,\epsilon,\sigma, K_M)$)
\be
|\bE \Phi_h - \|f_h\|_r^r| &\leq C_1 (nh\ln n)^{-\frac{r}{2}},\\
\var(\Phi_h) &\leq  h\cdot \int_0^1\frac{C_2}{n^{1-\epsilon}h}\left(|f_h(x)|^{2r-2}+\frac{1}{(nh)^{r-1}}\right)dx\\
&= \frac{C_2}{n^{1-\epsilon}}\left(\|f_h\|_{2r-2}^{2r-2}+\frac{1}{(nh)^{r-1}}\right).\label{eqn:lr_odd_phih_bias_var}
\ee

For the estimation performance of our final estimator $T_h$, first note that $\|f\|_p\le L$ implies
\begin{align*}
|T_h - \|f\|_r| \le |\max\{0, \Phi_h\}^{\frac{1}{r}} - \|f\|_r|. 
\end{align*}
Set $h=(n\ln n)^{-\frac{1}{2s+1}}$, and $\epsilon>0$ sufficiently small satisfying the conditions of Theorem \ref{theorem:lr_odd_bias_var}. We then divide our analysis into two cases.

First suppose that $\|f\|_r\le 2C_0h^s$, with $C_0$ defined in \eqref{eqn:besov_approx}.  Then $\|f_h\|_r\le \|f\|_r + \|f-f_h\|_r\le 3C_0h^s$, and by the bias bound of $\Phi_h$ in \eqref{eqn:lr_odd_phih_bias_var}, $|\bE \Phi_h|\leq C_3h^{sr}$ for a constant $C_3$ depending on $(C_0,C_1,r,s)$. Hence, by Lemma \ref{lemma:norm}, 
\be
\ & \left(\bE (T_h-\|f\|_r)^2\right)^\frac{1}{2}\\
&\leq \left(\bE (\max\{0, \Phi_h\}^{\frac{1}{r}}-\|f\|_r)^2\right)^\frac{1}{2} \\
&\leq 2\left( \sqrt{\bE [\max\{0, \Phi_h\}^{\frac{2}{r}}]} + \|f\|_r\right)\\
&\leq 2(\bE \Phi_h^2)^{\frac{1}{2r}} + 2C_0h^s\\
&\leq 2|\bE\Phi_h|^{\frac{1}{r}} + 2(\var(\Phi_h))^{\frac{1}{2r}} + 2C_0h^s\\
&\leq 2(C_3^{\frac{1}{r}}+C_0)h^{s}+
2\left[\frac{C_2}{n^{1-\epsilon}}\left(\|f_h\|_{2r-2}^{2r-2}+\frac{1}{(nh)^{r-1}}\right)\right]^{\frac{1}{2r}}\quad \text{(by \eqref{eqn:lr_odd_phih_bias_var})}
\\
&\leq 2(C_3^{\frac{1}{r}}+C_0)h^{s}+
2\left[\frac{C_2}{n^{1-\epsilon}}\left(\begin{array}{c}C(r,K_M)h^{-1+1/r}\|f\|_r^{r-1}\|f_h\|_r^{r-1}\\+\frac{1}{(nh)^{r-1}}\end{array}\right)\right]^{\frac{1}{2r}} \quad \text{(by Lemma \ref{lemma:norm})}\\
&\leq C_4\left[h^s + \left(\frac{h^{-1+1/r}\cdot h^{2s(r-1)}}{n^{1-\epsilon}}\right)^{\frac{1}{2r}} + \left(\frac{1}{n^{1-\epsilon}(nh)^{r-1}}\right)^{\frac{1}{2r}}\right]\leq 3C_4 h^s,\label{eqn:lr_odd_fmall}
\ee
where $C_4$ is a constant depending on $C_0,C_1,C_2,C_3, r,K_M,s$.

Second suppose that $\|f\|_r>2C_0h^s$, then $\|f_h\|_r\ge \|f\|_r - \|f-f_h\|_r>C_0h^s$. Using $|a^r-b^r|\ge b^{r-1}|a-b|$ for any $a,b\ge 0$ and $r\ge 1$, we have
\be
|T_h-\|f\|_r| &\le |T_h-\|f_h\|_r| + C_0h^s\\
&\le \frac{|T_h^r-\|f_h\|_r^r|}{\|f_h\|_r^{r-1}} + C_0h^s \\
&\le \frac{|\Phi_h-\|f_h\|_r^r|}{(C_0h^s)^{r-1}} + C_0h^s = C_0^{1-r}h^{(1-r)s}|\Phi_h-\|f_h\|_r^r| +C_0 h^s.
\ee
As a result, by triangle inequality we have
\be
\left(\bE (T_h-\|f\|_r)^2\right)^\frac{1}{2} &\leq C_0^{1-r}h^{(1-r)s}\left(\bE (\Phi_h-\|f\|_r^r)^2\right)^\frac{1}{2} + C_0h^s\\
&\leq C_0^{1-r}h^{(1-r)s}|\bE \Phi_h-\|f_h\|_r^r| + C_0^{1-r}h^{(1-r)s}\sqrt{\var(\Phi_h)} + C_0h^s \\
&\leq C_4\left[\begin{array}{c}h^{(1-r)s}\cdot h^{sr} + h^{(1-r)s}\sqrt{n^{\epsilon-1}h^{-1+1/r}h^{2s(r-1)}}\\
	+h^{(1-r)s}\sqrt{n^{\epsilon-1}(nh)^{-(r-1)}} + h^s\end{array}\right]\leq C_5 h^s,\\\label{eqn:lr_odd_flarge}
\ee
where the second inequality in the above display follows similar to before by equation \eqref{eqn:lr_odd_phih_bias_var} and Lemma \ref{lemma:norm} with $C_4$ being a constant depending on $C_0,C_1,C_2,C_3, r,K_M,s$. Combining inequalities \eqref{eqn:lr_odd_fmall} and \eqref{eqn:lr_odd_flarge} completes the proof of Theorem \ref{theorem:lr_odd_bias_var}. \qed

\subsection{Proof of Theorem \ref{theorem:lone_lower_bound} and Theorem \ref{theorem:lr_odd_lower_bound}}
The outline for the proof of lower bounds (i.e., Theorems \ref{theorem:lone_lower_bound} and \ref{theorem:lr_odd_lower_bound}) is as follows. In Section \ref{subsec.parametric_reduction}, we reduce the nonparametric problem to a parametric subproblem in the Gaussian location model. The minimax lower bound for the parametric submodel is proved by a generalized version of Le Cam's method involving a pair of priors, also known as the method of two fuzzy hypotheses \cite{Tsybakov2008}. In Section \ref{subsec.measure_construction} we construct the priors using duality to best approximation and Section \ref{subsec.minimax_lower_bound} finishes the proof. 

In the sequel we assume that $r\ge 1$ is a fixed non-even real number. 

\subsubsection{Reduction to Parametric Submodel}\label{subsec.parametric_reduction}
Fix a smooth function $g(x)$ vanishing outside $[0,1]$ with $\|g\|_{\cB_{p,\infty}^s}=1$, and $\|g\|_1>0$. Set 
\be
h&= (n\ln n)^{-\frac{1}{2s+1}},\\
N&=h^{-1},
\ee
and let $I_1,\cdots,I_N$ be the partition of the interval $[0,1]$ into $N$ subintervals of length $h$ each (without loss of generality we assume that $N$ is an integer), and let $t_i$ be the left endpoint of subinterval $I_i$. With a point $\theta=(\theta_1,\cdots,\theta_N)\in [-\sqrt{\ln N}, \sqrt{\ln N}]^N$ we associate the function
\be
f_\theta(t) = L'\sum_{i=1}^N \theta_i\sqrt{\ln N}\cdot h^sg\left(\frac{t-t_i}{h}\right).
\ee

\begin{lemma}\label{lemma:besov_ball}
	If 
	\be \label{eq.constr_1}
	\theta\in\Theta\triangleq [-\sqrt{\ln N}, \sqrt{\ln N}]^N\cap \left\{\theta: \frac{1}{N}\sum_{i=1}^N |\theta_i|^p \le \left(\frac{2}{\sqrt{\ln N}}\right)^p\right\}, 
	\ee
	then for some constant $L'>0$ independent of $n$, we have $f_\theta(t)\in \cB_{p,\infty}^s(L)$. 
\end{lemma}
\begin{proof}
	Let $s_0 = \lfloor s\rfloor + 1$. Observe that
	\begin{align*}
	\|f_\theta\|_p = L'\|g\|_ph^s\cdot \sqrt{\ln N}\left(\frac{1}{N}\sum_{i=1}^N |\theta_i|^p\right)^{\frac{1}{p}},
	\end{align*}
	the condition \eqref{eq.constr_1} ensures that $h^{-s}\|f_\theta\|_p$ is upper bounded by a numerical constant proportional to $L'$. By \eqref{eq:Kequiv2} in Lemma \ref{lmm:Kequiv}, this implies that there exists a constant $C_0$ independent of $n$ such that $\omega_{s_0}(f,t^r)_p\le C_0L't^{s_0}$ for any $t\ge h$. Moreover, 
	\begin{align*}
	\|f_\theta^{(s_0)}\|_p = L'\|g^{(s_0)}\|_ph^{s-s_0}\cdot \sqrt{\ln N}\left(\frac{1}{N}\sum_{i=1}^N |\theta_i|^p\right)^{\frac{1}{p}},
	\end{align*}
	the condition \eqref{eq.constr_1} ensures that $h^{s_0-s}\|f_\theta^{(s_0)}\|_p$ is upper bounded by a numerical constant proportional to $L'$. By \eqref{eq:Kequiv1} in Lemma \ref{lmm:Kequiv}, this implies that there exists a constant $C_0$ independent of $n$ such that $\omega_{s_0}(f,t^r)\le C_0L't^{s_0}$ for any $t\le h$. Now by the definition of the Besov norm in \eqref{eq.besov_norm}, a suitable choice of the scale parameter $L'$ ensures that $f_\theta(t)\in \cB_{p,\infty}^s(L)$.
\end{proof}

Fix any choice of $L'$ given by Lemma \ref{lemma:besov_ball}. Note that for any $r\ge 1$, we have
\be\label{eq.nonpara_exp}
\|f_\theta\|_r = L'\|g\|_r(n\ln n)^{-\frac{s}{2s+1}}\cdot \sqrt{\ln N}\left(\frac{1}{N}\sum_{i=1}^N |\theta_i|^r\right)^{\frac{1}{r}}.
\ee
Hence, a sufficient condition for Theorem \ref{theorem:lone_lower_bound} and Theorem \ref{theorem:lr_odd_lower_bound} is that
\be\label{eq.para_lower}
\inf_{\hat{T}}\sup_{\theta\in\Theta} \bE_\theta\left( \hat{T} - \left(\frac{1}{N}\sum_{i=1}^N |\theta_i|^r\right)^{\frac{1}{r}}\right)^2 \gtrsim \left(\frac{1}{\sqrt{\ln N}}\right)^2 = \frac{1}{\ln N},
\ee
where the infimum is taken over all estimators $\hat{T}$ which is a measurable real-valued function of $\{Y(t)\}_{t\in [0,1]}$. 

Finally, we note that by the factorization theorem, to estimate $\left(\frac{1}{N}\sum_{i=1}^N |\theta_i|^r\right)^{1/r}$ for $\theta\in\Theta$, the vector $y=(y_1,\cdots,y_N)$ with
\be
y_i \triangleq \frac{\sqrt{n}}{\sigma\|g\|_2\sqrt{h}}\int_{I_i} g\left(\frac{t-t_i}{h}\right)dY(t), \qquad i=1,\cdots,N,
\ee
constitute a sufficient statistic for the Gaussian white noise model. Note that
\be\label{eq.submodel}
y_i = \alpha \theta_i + \xi_i, \qquad i=1,\cdots,N,
\ee
with $\theta\in\Theta$, and 
\be
\alpha &\triangleq \sigma^{-1}L' \|g\|_2 n^{1/2}h^{s+1/2}\cdot \sqrt{\ln N} \asymp 1,\\
\xi_i &\triangleq \frac{1}{\|g\|_2\sqrt{h}}\int_{I_i}g\left(\frac{t-t_i}{h}\right)dB_t.
\ee
As a result, $\xi_1,\cdots,\xi_N$ are i.i.d. $\cN(0,1)$ random variables. Hence, we may further assume that our observation model is 
\begin{align}
y_i \overset{\text{ind}}{\sim} \cN(\alpha \theta_i, 1), \qquad i=1,\cdots,N,
\end{align}
which is a Gaussian location model with $\theta\in \Theta$, and the estimator $\hat{T}$ in \eqref{eq.para_lower} is a function of $(y_1,\cdots,y_N)$. Note that when $r=1$, this parametric subproblem is similar to but very \emph{different} from the problem considered in \cite{cai2011testing}, where the authors did not have the second constraint in (\ref{eq.constr_1}).

\subsubsection{Construction of Two Priors}\label{subsec.measure_construction}
The minimax lower bound \eqref{eq.para_lower} follows from the generalized Le Cam's method involving two priors, which is known as the method of two fuzzy hypotheses presented in \cite{Tsybakov2008}. Suppose we observe a random vector ${\bf Z} \in (\mathcal{Z},\mathcal{A})$ which has distribution $P_\theta$ where $\theta \in \Theta$. Let $\sigma_0$ and $\sigma_1$ be two prior distributions supported on $\Theta$. Write $F_i$ for the marginal distribution of $\mathbf{Z}$ when the prior is $\sigma_i$ for $i = 0,1$. 
Let $\hat{T} = \hat{T}({\bf Z})$ be an arbitrary estimator of a function $T(\theta)$ based on $\bf Z$, and $V(P,Q)$ be the total variation distance between two probability measures $P,Q$ on the measurable space $(\mathcal{Z},\mathcal{A})$. Concretely,
\be
V(P,Q) \triangleq \sup_{A\in \mathcal{A}} | P(A) - Q(A) | = \frac{1}{2} \int |p-q| d\nu,
\ee
where $p = \frac{dP}{d\nu}, q = \frac{dQ}{d\nu}$, and $\nu$ is a dominating measure so that $P \ll \nu, Q \ll \nu$. We have the following general minimax lower bound.

\begin{lemma}\cite[Theorem 2.15]{Tsybakov2008} \label{lemma.tsybakov}
	Suppose that there exist $\zeta\in \mathbb{R}, \delta>0, 0\leq \beta_0,\beta_1 <1$ such that
	\be
	\sigma_0(\theta: T(\theta) \leq \zeta -\delta) & \geq 1-\beta_0, \\
	\sigma_1(\theta: T(\theta) \geq \zeta + \delta) & \geq 1-\beta_1.
	\ee
	If $V(F_1,F_0) \leq \eta<1$, then
	\be
	\inf_{\hat{T}} \sup_{\theta \in \Theta} \bP_\theta\left( |\hat{T} - T(\theta)| \geq \delta \right) \geq \frac{1-\eta - \beta_0 - \beta_1}{2},
	\ee
	where $F_i,i = 0,1$ are the marginal distributions of $\mathbf{Z}$ when the priors are $\sigma_i,i = 0,1$, respectively.
\end{lemma}

In the remainder of this section, we aim to construct two priors $\mu_0, \mu_1$ supported on $[-\sqrt{\ln N}, \sqrt{\ln N}]$ such that the following conditions hold (the numerical constant $d>0$ is chosen later): 
\begin{align}
&\int t^l \mu_1(dt) = \int t^l \mu_0(dt), \qquad \text{for all } l=0,1,\cdots,K\triangleq \lceil d\ln N\rceil, \label{eq.moment_match} \\
&\int |t|^r \mu_1(dt) - \int |t|^r \mu_0(dt) \gtrsim (\ln N)^{-\frac{r}{2}}, \label{eq.separation} \\
&\int |t|^p \mu_i(dt) \le (\ln N)^{-\frac{p}{2}}, \qquad \text{for }i=0,1. \label{eq.moment_bound}
\end{align}
In the next section we will choose the priors $\sigma_i, i=0,1$ in Lemma \ref{lemma.tsybakov} to be close to the product measure $\mu_i^{\otimes N}, i=0,1$ with each copy given above. The condition \eqref{eq.moment_match} ensures a small total variation distance $V(F_1,F_0)$ in Lemma \ref{lemma.tsybakov}, the condition \eqref{eq.separation} ensures a large $\delta\asymp (\ln N)^{-\frac{r}{2}}$ in Lemma \ref{lemma.tsybakov}, and the condition \eqref{eq.moment_bound} ensures that the support of $\mu_i^{\otimes N}$ is almost supported on $\Theta$ given in \eqref{eq.constr_1}. 

The following result is simply the duality between the problem of best uniform approximation and moment matching. 
\begin{lemma}\label{lem.measure}
	For any bounded interval $I\subseteq \bR$ not containing zero, integers $q\ge 0,K>0$ and continuous function $f$ on $I$, let 
	\[
	E_{q-1,K}(f;I) \triangleq \inf_{\{a_i\}} \sup_{x\in I} \left| \sum_{i=-q+1}^K a_i x^i
	-f(x)\right| 
	\]
	denote the best uniform approximation error of $f$ by rational functions spanned by $\{x^{-q+1},x^{-q+2},\cdots,x^K\}$.
	Then
	\begin{equation}
	\begin{aligned}
	2 E_{q-1,K}(f;I) = \max & ~ \int f(t) \nu_1(dt) - \int f(t) \nu_0(dt),   \\
	\text{\rm s.t.}     & ~ \int t^{l} \nu_1(dt) = \int t^{l} \nu_0(dt), \quad l=-q+1,\cdots,K,
	\end{aligned}
	\label{eq:Rstar}
	\end{equation}
	where the maximum is taken over pairs of probability measures $\nu_0$ and $\nu_1$ supported on $I$.
\end{lemma}

Lemma \ref{lem.measure} extends the duality results in \cite[Lemma 1]{cai2011testing}, \cite[Lemma 10, Lemma 12]{jiao2015minimax} to rational functions. We list two possible proofs of Lemma \ref{lem.measure}. The first proof relies on the Hahn--Banach theorem and the Riesz representation theorem, where the essential argument is given in \cite{lepski1999estimation}. The second proof makes use of the fact that the rational functions $\{x^{-q+1}, x^{-q+2},\cdots, x^K \}$ form a Chebyshev system in $C(I)$ and therefore the Chebyshev alternation theorem holds \cite[Chapter 3, Theorem 5.1]{Devore--Lorentz1993}, so that the probability measures $\nu_0$ and $\nu_1$ can be explicitly constructed following the similar lines to \cite[Appendix E]{wu2016minimax}. For completeness we include the second proof in Section \ref{subsec:measure_proof}. 

Here we apply this lemma to $f_q(t)=t^{-q+\frac{r}{2}}$ and
\be
K=\lceil d\ln N \rceil, \qquad I = \left[\frac{c}{(\ln N)^2}, 1\right], \qquad  q=\left\lceil\frac{p}{2}\right\rceil, 
\ee
where the constant $c\in (0,1)$ appearing in the definition of $I$ is given by the following lemma, which provides a lower bound for the approximation error of $f_q(t)$.
\begin{lemma}\label{lem.approx}
	Fix a non-even real $r\ge 1$ and an integer $q\ge r/2$. For $f_q(t)=t^{-q+\frac{r}{2}}$, we have
	\be
	\liminf_{n\to\infty} n^{-(2q-r)}E_{q-1,n}\left(f;[\frac{c}{n^2},1]\right) \ge c', 
	\ee
	where constants $c\in (0,1), c'>0$ only depend on $(q,r)$. 
\end{lemma}
By Lemma \ref{lem.approx} and our definitions of $f$, $I$ and $K$, we conclude that
\be
E_{q-1,K}(f;I) \gtrsim (\ln N)^{2q-r}.
\ee
Let $\nu_0, \nu_1$ be the maximizers of \eqref{eq:Rstar}. We define probability measures $\tilde{\nu}_0, \tilde{\nu}_1$ by
\be
\tilde{\nu}_i(dx) = \left[1-\bE_{X\sim \nu_i} \left(\frac{c^q}{(\ln N)^{2q}X^q}\right) \right]\delta_0(dx) + \left(\frac{c}{(\ln N)^2x}\right)^{q} \nu_i(dx)\quad i=0,1,
\ee
where $\delta_0(\cdot)$ is the delta measure at zero. It is straightforward to verify that $\tilde{\nu}_i$ forms a probability measure supported on $[0,1]$, and
\begin{enumerate}
	\item $\int t^l\tilde{\nu}_1(dt) = \int t^l \tilde{\nu}_0(dt)$, for all $l=0,1,\cdots,q+K$;
	\item $\int t^{\frac{r}{2}}\tilde{\nu}_1(dt) - \int t^{\frac{r}{2}}\tilde{\nu}_0(dt) \gtrsim (\ln N)^{-r}$;
	\item $\int t^{q}\tilde{\nu}_i(dt)= c^q(\ln N)^{-2q}\le (\ln N)^{-2q}$, for $i=0,1$.
\end{enumerate}

Finally, we define the measures $\mu_0, \mu_1$ as follows. For $i=0,1$, let $X_i$ follow the distribution $\tilde{\nu}_i$, the measure $\mu_i$ is defined as the probability distribution of $\epsilon_i\sqrt{X_i\cdot \ln N}$, where $\epsilon_i\sim\mathsf{Unif}(\{\pm1\})$ is independent of $X_i$. Clearly the measures $\mu_0, \mu_1$ are supported on $[-\sqrt{\ln N}, \sqrt{\ln N}]$, and it remains to check the conditions \eqref{eq.moment_match} to \eqref{eq.moment_bound}. 

For \eqref{eq.moment_match}, since $\mu_i$ is symmetric around zero, the condition clearly holds for odd $l$. For even $l=2k$ with $0\le k\le K$, by the first property of $\tilde{\nu}_i$ we have
\begin{align*}
\int t^l \mu_1(dt) &= \int (\epsilon_i \sqrt{t\cdot \ln N})^{2k} \mu_1(dt) = (\ln N)^k \int t^k\mu_1(dt) \\
&= (\ln N)^k \int t^k\mu_0(dt) = \int (\epsilon_i \sqrt{t\cdot \ln N})^{2k} \mu_0(dt) = \int t^l \mu_0(dt), 
\end{align*}
i.e., \eqref{eq.moment_match} holds. Similarly, the condition \eqref{eq.separation} is checked via
\begin{align*}
\int |t|^r \mu_1(dt) - \int |t|^r \mu_0(dt) = (\ln N)^{\frac{r}{2}} \int t^{\frac{r}{2}}(\tilde{\nu}_1(dt) - \tilde{\nu}_2(dt)) \gtrsim (\ln N)^{-\frac{r}{2}}. 
\end{align*}
Finally, for \eqref{eq.moment_bound}, first note that
\begin{align*}
\int t^{2q}\mu_i(dt) = (\ln N)^q \int t^q \tilde{\nu}_i(dt) \le (\ln N)^{-q}, \qquad i=0,1.
\end{align*}
Since $2q\ge p$, H\"{o}lder's inequality yields
\be
\left(\int |t|^p \mu_i(dt) \right)^{\frac{1}{p}} \le \left(\int |t|^{2q} \mu_i(dt) \right)^{\frac{1}{2q}}\le \frac{1}{\sqrt{\ln N}},
\ee
i.e., \eqref{eq.moment_bound} holds. Hence, the construction of $\mu_0,\mu_1$ satisfies all conditions in \eqref{eq.moment_match} to \eqref{eq.moment_bound}. This construction is partially inspired by \cite{wu2016minimax}. We remark that the construction heavily relies on the fact that $p$ is finite, where for $p=\infty$, Lemma \ref{lem.approx} fails and the condition \eqref{eq.moment_bound} would require that the priors $\mu_0, \mu_1$ be supported on a smaller interval $[-\frac{1}{\sqrt{\ln N}}, \frac{1}{\sqrt{\ln N}}]$.

%
%

\subsubsection{Minimax Lower Bound in the Parametric Submodel}\label{subsec.minimax_lower_bound}
In this section we invoke Lemma \ref{lemma.tsybakov} to finish the proof of \eqref{eq.para_lower}, thereby proving the lower bound in Theorems \ref{theorem:lone_lower_bound} and \ref{theorem:lr_odd_lower_bound}. Consider the probability measures $\mu_0,\mu_1$ constructed in the Section \ref{subsec.measure_construction}, and define
\be
\Delta = \int |t|^r\mu_1(dt) - \int |t|^r\mu_0(dt) \asymp (\ln N)^{-\frac{r}{2}}. 
\ee
Denote by $\mu_i^{\otimes N}$ the $N$-fold product of $\mu_i$. Consider the following event:
\be
E_i \triangleq \{\theta: \theta\in\Theta\} \cap \left\{\theta: \left|\frac{1}{N}\sum_{j=1}^N|\theta_j|^r - \bE_{\mu_i} |\theta|^r\right| \le \frac{\Delta}{4} \right\}, \qquad i=0,1. 
\ee
By Chebyshev's inequality it is easy to show that for $i=0,1$, 
\be
\mu_i^{\otimes N}(\left\{ \theta:\theta\notin \Theta \right\}) &= \mu_i^{\otimes N} \left(\left\{\theta: \frac{1}{N}\sum_{j=1}^N |\theta_j|^p > \left(\frac{2}{\sqrt{\ln N}}\right)^p\right\}\right)\\
&\le \mu_i^{\otimes N} \left(\left\{\theta: \frac{1}{N}\sum_{j=1}^N |\theta_j|^p - \bE_{\mu_i} |\theta|^p> \left(\frac{1}{\sqrt{\ln N}}\right)^p\right\}\right)\\
&\le \left(\frac{1}{\sqrt{\ln N}}\right)^{-2p}\cdot \mathsf{Var}_{\mu_i^{\otimes N}}\left(\frac{1}{N}\sum_{j=1}^N |\theta_j|^p\right)\\
&= \frac{1}{N}\left(\frac{1}{\sqrt{\ln N}}\right)^{-2p}\cdot \mathsf{Var}_{\mu_i}(|\theta|^p) \\
&\le \frac{1}{N}(\frac{1}{\sqrt{\ln N}})^{-2p}\cdot (\sqrt{\ln N})^{2p} \to 0, 
\ee
and
\be
\mu_i^{\otimes N}\left(\left\{\theta: \left|\frac{1}{N}\sum_{j=1}^N|\theta_j|^r - \bE_{\mu_i} |\theta|^r\right| > \frac{\Delta}{4} \right\}\right) &\le \frac{16}{\Delta^2}\mathsf{Var}_{\mu_i^{\otimes N}}\left(\frac{1}{N}\sum_{j=1}^N |\theta_j|^r\right)\\
&\le \frac{16}{N\Delta^2}\cdot (\sqrt{\ln N})^{2r} \to 0.
\ee

Hence, by the union bound, we have
\be
\mu_i^{\otimes N}(E_i^c) \le \mu_i^{\otimes N}(\left\{ \theta:\theta\notin \Theta \right\}) + \mu_i^{\otimes N}\left(\left\{\theta: \left|\frac{1}{N}\sum_{j=1}^N|\theta_j|^r - \bE_{\mu_i} |\theta|^r\right| > \frac{\Delta}{4} \right\}\right) \to 0
\ee
for any $i=0,1$. 

Now we are ready to apply Lemma \ref{lemma.tsybakov} to
\be
T(\theta) &= \left(\frac{1}{N}\sum_{i=1}^N |\theta_i|^r \right)^\frac{1}{r},\\
\zeta &= \left( \frac{\bE_{\mu_1^{\otimes N}} [T(\theta)^r] + \bE_{\mu_0^{\otimes N}} [T(\theta)^r]}{2}\right)^{\frac{1}{r}},
\ee
and let the prior $\sigma_i$ be the conditional distribution of $\mu_i^{\otimes N}$ conditioning on $E_i$, i.e., 
\be
\sigma_i(\cdot) = \frac{\mu_i^S(\cdot\cap E_i)}{\mu_i^S(E_i)}, \qquad i=0,1.
\ee

By definition of $E_i$, the measure $\sigma_i$ is a valid prior on $\Theta$. Moreover, under $\sigma_1$ we have
\be
T(\theta)^r - \zeta^r \ge \frac{\Delta}{4}.
\ee
By definition of $E_1$ and $\Theta$, under $\sigma_1$ we have
\be
T(\theta) &= \left(\frac{1}{N}\sum_{i=1}^N |\theta_i|^r \right)^\frac{1}{r} \le \left(\frac{1}{N}\sum_{i=1}^N |\theta_i|^p \right)^\frac{1}{p}\le \frac{2}{\sqrt{\ln N}},\\
\zeta &\le \frac{1}{2}\left(\int |t|^r \mu_1(dt)\right)^\frac{1}{r} + \frac{1}{2}\left(\int |t|^r \mu_0(dt)\right)^\frac{1}{r} \\
& \le \frac{1}{2}\left(\int |t|^p \mu_1(dt)\right)^\frac{1}{p} + \frac{1}{2}\left(\int |t|^p \mu_0(dt)\right)^\frac{1}{p} \le \frac{1}{\sqrt{\ln N}}.
\ee
Hence, using the inequality $a^r-b^r\le r(a^{r-1}+b^{r-1})(a-b)$ for any $a\ge b>0$ and $r\ge 1$, the previous inequalities yield that under $\sigma_1$, 
\be
T(\theta) - \zeta \ge \frac{[T(\theta)]^r - \zeta^r}{r([T(\theta)]^{r-1}+\zeta^{r-1})} \gtrsim \frac{\Delta}{(\ln N)^{-\frac{r-1}{2}}} \gtrsim \frac{1}{\sqrt{\ln N}}.
\ee
Similarly, under $\sigma_0$ we have
$
T(\theta) - \zeta \lesssim - \frac{1}{\sqrt{\ln N}}, 
$
hence in Lemma \ref{lemma.tsybakov} we can set
\be
\delta \gtrsim \frac{1}{\sqrt{\ln N}}, 
\ee
so that $\beta_0=\beta_1=0$. 

Now denote by $F_0,F_1$ the marginal distributions of $\mathbf{Z}$ based on priors $\sigma_0,\sigma_1$, and the counterparts $G_0,G_1$ based on priors $\mu_0^{\otimes N},\mu_1^{\otimes N}$. By the data-processing property of the total variation distance, we have
\be
V(F_i,G_i) \le V(\sigma_i, \mu_i^{\otimes N}) = \mu_i^{\otimes N}(E_i^c) \to 0, \qquad i=0,1.
\ee
Moreover, \cite{cai2011testing} shows that the $\chi^2$ distance between $G_0$ and $G_1$ is upper bounded as
\be
\chi^2(G_0,G_1) \le \left(1 + e^{3\alpha^2\ln N/2}\left(\frac{\alpha e\ln N}{d\ln N}\right)^{d\ln N}\right)^N - 1. 
\ee
Hence, for choosing $d>0$ large enough, $\chi^2(G_0,G_1)$ is upper bounded by a universal constant $C$. Now by Lemma \ref{lemma:TV_chi}, we have
\be
V(G_0,G_1) \le 1-\frac{1}{2}\exp(-C). 
\ee

In summary, the triangle inequality for total variation distance yields
\be
V(F_0,F_1) &\le V(F_0,G_0) + V(G_0,G_1) + V(G_1,F_1) \to 1-\frac{1}{2}\exp(-C)  < 1, 
\ee
and Lemma \ref{lemma.tsybakov} together with Markov's inequality yields
\be
\inf_{\hat{T}}\sup_{\theta\in\Theta} \bE_\theta\left( \hat{T} - \left(\frac{1}{N}\sum_{i=1}^N |\theta_i|^r\right)^{\frac{1}{r}}\right)^2 &\gtrsim \delta^2\cdot \inf_{\hat{T}}\sup_{\theta\in\Theta} \bP\left(|\hat{T}-T(\theta)|\ge \delta\right)\\
&\gtrsim \frac{1}{\ln N}\cdot \frac{1}{4}\exp(-C)\gtrsim \frac{1}{\ln N}, 
\ee
which is (\ref{eq.para_lower}), as desired. \qed

\subsection{Proof of Theorem \ref{theorem:lr_odd_adaptive}}
For $s\in [0,s_{\max}]$ define the ideal bandwidth $h^*: [0,s_{\max}]\rightarrow \mathcal{H}$ by
\be 
h^*(s)\triangleq \frac{1}{\lfloor (n\ln{n}) ^{\frac{1}{2s+1}}\rfloor}.
\ee
Then
\be 
\ & \E\left(\hat{T}-\|f\|_r\right)^2\\
&=\E\left[\left(\hat{T}-\|f\|_r\right)^2\mathbbm{1}\left(\hat{h}\geq h^*(s)\right)\right]+\E\left[\left(\hat{T}-\|f\|_r\right)^2\mathbbm{1}\left(\hat{h}< h^*(s)\right)\right]\\
&=\text{I}+\text{II}.
\ee
First note that
\be 
\text{I} &\leq 2\left\{\E\left[\left(T_{h^*(s)}-\|f\|_r\right)^2\right]+\E\left[\left({T}_{\hat{h}}-{T}_{h^*(s)}\right)^2\mathbbm{1}\left(\hat{h}\geq h^*(s)\right)\right]\right\}\\
&\leq 2\left\{C_0(n\ln{n})^{-\frac{2s}{2s+1}}+C^*\frac{\lambda_{h^*(s)}^2}{\ln n}\right\} \le C_0^*(n\ln{n})^{-\frac{2s}{2s+1}} ,\label{eqn:lr_even_lepski_no_error}
\ee
where the last line follows from the definitions of $h^*(s)$ and $\hat{h}$ respectively, and $C_0$ is a constant that depends on $(c_1,c_2,\epsilon,s_{\max},\sigma,L)$. 


To upper bound $\text{II}$, we have
\be 
\text{II}&=\E\left[\left(\hat{T}-\|f\|_r\right)^2\mathbbm{1}\left(\hat{h}< h^*(s)\right)\right] \le L^2\cdot \bP(\hat{h} < h^*(s))\\
&\le L^2\cdot \sum_{h<h^*(s):h\in \mathcal{H}}\P\left(|T_h-T_{h^*(s)}|>\lambda_h\sqrt{\frac{C^*}{\ln n}}\right),\label{eqn:lr_even_lepski_term_II_prob}
\ee
where the first inequality is due to that $\hat{T}, \|f\|_r\in [0,L]$, and the second inequality follows from the fact that $h^*(s)$ is not a feasible condidate in \eqref{eq.adaptation}. Let $$\gamma_n(h)\triangleq \lambda_h\sqrt{\frac{C^*}{\ln n}},$$
below we reduce the control of $\P\left(|T_h-T_{h^*(s)}|>\gamma_n(h)\right)$ to suitable controls over $$\E\left|\int_{0}^1(T_h(x)-\E T_h(x))dx\right|^k$$ for some $k$ to be chosen a large enough constant depending on the tuple $(s_{\max},r,\sigma,p)$. Consequently, in the following lemma we demonstrate the desired control over central moments of $\int_{0}^{1}T_h(x)dx$ for every $h\in \mathcal{H}$.

\begin{lemma}\label{lemma:lr_odd_phih_moments}
	Let $f\in \cB_{p,\infty}^s(L)$ and $r\ge 1$ be non-even. Then, under the assumptions of Theorem \ref{theorem:lr_odd_adaptive}, for any $h\in \mathcal{H}$ and integer $k\geq 2$, 
	\be 
	\ & \E\left|\int_{0}^1(T_h(x)-\E T_h(x))dx\right|^k\\
	&\leq C(r,k,c_1,\sigma)n^{2k\epsilon}\left(\begin{array}{c}\lambda_h^{kr}(h^{k-1}+h^{\frac{k}{2}}) \\+ \lambda_h^k(h^{k-1}\|f_h\|_{k(r-1)}^{k(r-1)} + h^{\frac{k}{2}}\|f_h\|_{2(r-1)}^{k(r-1)})\end{array}\right),
	\ee 
	where $C(r,k,c_1,\sigma)$ is a constant depending on $(r,k,c_1,\sigma)$.
\end{lemma}

By Lemma \ref{lemma:norm} and $h\in [0,1]$, the following corollary is immediate. 
\begin{corollary}\label{cor:lr_odd_phih_moments}
	Under the assumptions of Theorem \ref{theorem:lr_odd_adaptive}, for any $h\in \mathcal{H}$ and integer $k\ge 2$, 
	\be 
	\ & \E\left|\int_{0}^1(T_h(x)-\E T_h(x))dx\right|^k\\
	&\leq C'(r,k,c_1,\sigma)n^{2k\epsilon}\left(\begin{array}{c}\lambda_h^{kr}h^{\frac{k}{2}} + \lambda_h^k(h^{\frac{k-1}{r}}\|f\|_r^{(k-1)(r-1)}\|f_h\|_{r}^{r-1} \\+ h^{\frac{k}{2r}}\|f\|_r^{k(r-1)/2}\|f_h\|_{r}^{k(r-1)/2})\end{array}\right),
	\ee 
	where $C'(r,k,c_1,\sigma)$ is a constant depending on $(r,k,c_1,\sigma)$.
\end{corollary}

We defer the proof of Lemma \ref{lemma:lr_odd_phih_moments} to Section \ref{section:technical_lemmas} and complete the proof of Theorem \ref{theorem:lr_odd_adaptive} assuming its validity. Recall that $\|f-f_h\|_r\le C_1h^s$ for some $C_1=C_1(s,r,L,K_M)$ for any $f\in \cB_{p,\infty}^s(L)$. We divide the subsequent analysis into two cases. First consider the case when $\|f_h\|_r\leq \frac{C_1}{\sqrt{nh\ln n}}$, then
\be
\P\left(|T_h-T_{h^*(s)}|>\gamma_n(h)\right)
&\leq \P\left(T_h>\frac{\gamma_n(h)}{2}\right)+\P\left(T_{h^*(s)}>\frac{\gamma_n(h)}{2}\right)\\
&\leq 
\P\left(T_h>\frac{\gamma_n(h)}{2}\right)+\P\left(T_{h^*(s)}>\frac{\gamma_n(h^*(s))}{2}\right).
\label{eqn:lr_odd_lepski_error_decomp_fhsmall}
\ee
where the last inequality uses the monotone decreasing nature of $h\mapsto \gamma_n(h)$. By \eqref{eqn:lr_odd_phih_bias_var} and the triangle inequality, we have
\begin{align*}
\int_0^1 \bE T_h(x)dx \le \|f_h\|_r^r + \left|\int_0^1 \bE T_h(x)dx - \|f_h\|_r^r \right| \le C_2(nh\ln n)^{-\frac{r}{2}},
\end{align*}
where $C_2>0$ is a constant depending on $(c_1,c_2,r,C_1)$. Choosing $C^*$ large enough such that
\begin{align*}
\left(\frac{\gamma_n(h)}{4}\right)^r \ge C_2(nh\ln n)^{-\frac{r}{2}},
\end{align*}
for any $k\ge 2$ we have
\be 
\  \P\left(T_h>\frac{\gamma_n(h)}{2}\right)
&\le \P\left(\int_{0}^1T_h(x)dx>\left(\frac{\gamma_n(h)}{2}\right)^r\right)\\
&\le \P\left(\int_{0}^1T_h(x)dx-\E\left(\int_{0}^1T_h(x)dx\right)> \left(\frac{\gamma_n(h)}{2}\right)^r-C_2(nh\ln n)^{-\frac{r}{2}}\right)\\
&\leq \P\left(\left|\int_{0}^1T_h(x)dx-\E\left(\int_{0}^1T_h(x)dx\right)\right|> \left(\frac{\gamma_n(h)}{4}\right)^r\right)\\
&\leq \frac{4^{rk}\E\left|\int_{0}^1T_h(x)dx-\E\left(\int_{0}^1T_h(x)dx\right)\right|^k}{(\gamma_n(h))^{kr}}.
\ee
However, note that
\begin{align*}
\|f\|_r \le \|f-f_h\|_r + \|f_h\|_r \le C_1h^s + \frac{C_1}{\sqrt{nh\ln n}} \le C_3\gamma_n(h),
\end{align*}
where $C_3$ is a constant depending on $(C_1,C^*)$, Corollary \ref{cor:lr_odd_phih_moments} yields
\be 
\ &  \E\left|\int_{0}^1(T_h(x)-\E T_h(x))dx\right|^k
\\ 
&\le C_4n^{2k\epsilon}\left[\lambda_h^{kr}h^{\frac{k}{2}} + \lambda_h^k h^{\frac{k-1}{r}} \gamma_n(h)^{k(r-1)} + \lambda_h^k h^{\frac{k}{2r}} \gamma_n(h)^{k(r-1)}\right] \\
&\le C_5n^{2k\epsilon}(\ln n)^{\frac{k}{2}}\cdot h^{\frac{k}{2r}}\gamma_n(h)^{kr}, 
\ee
where $C_4, C_5$ are constants depending on $(r,k,c_1,\sigma,C_1,C_3,C^*)$. Consequently, we have shown that for any $h\le h^*(s)$, 
\be
\bP(|T_h - T_{h^*(s)}|> \gamma_n(h)) \le 4^{kr}C_5n^{2k\epsilon}(\ln n)^{\frac{k}{2}}\cdot h^{\frac{k}{2r}}. \label{eqn:lr_odd_lepski_fhsmall_final}
\ee

Next consider the case when $\|f_h\|_r>\frac{C_1}{\sqrt{nh\ln n}}$. Note that
\begin{align*}
|\|f_{h^*(s)}\|_r-\|f_h\|_r|\leq \|f_{h^*(s)} - f\|_r + \|f_h-f\|_r \le 2C_1(h^*(s))^r, 
\end{align*}
thus if $C^*$ is large enough such that $\gamma_n(h)\ge 6C_1(h^*(s))^r$, triangle inequality yields
\be
\ &\P\left(|T_h-T_{h^*(s)}|>\gamma_n(h)\right)\\
&\leq \P\left(|T_h-\|f_h\|_r|>\frac{\gamma_n(h)}{3}\right) +\P\left(|T_{h^*(s)}-\|f_{h^*(s)}\|_r|>\frac{\gamma_n(h)}{3}\right). \label{eqn:lr_odd_lepski_error_decomp_fhlarge}
\ee
In this case, once again using $|a^r-b^r|\ge b^{r-1}|a-b|$ for any $a,b\ge 0$ and $r\ge 1$,
\be
|T_h-\|f_h\|_r| &\le \frac{|T_h^r-\|f_h\|_r^r|}{\|f_h\|_r^{r-1}}\leq \frac{|\int_{0}^1 T_h(x)dx-\|f_h\|_r^r|}{\|f_h\|_r^{r-1}}, \label{eqn:lr_odd_fh_large}
\ee
where the last inequality follows from $0\le \|f_h\|_r\le L$.
%
Fix any $h\leq  h^*(s)$, \eqref{eqn:lr_odd_fh_large} yields
\be 
\ & \P\left(|T_h-\|f_h\|_r|>\frac{\gamma_n(h)}{3}\right)\\
&\le \P\left(\left|\int_{0}^1 T_h(x)dx-\|f_h\|_r^r\right|>\frac{\gamma_n(h)}{3}\|f_h\|_r^{r-1}\right)\\
&\leq \P\left(\left|\int_{0}^1 T_h(x)dx-\E\left(\int_{0}^1 T_h(x)dx\right)\right|> \frac{\gamma_n(h)}{3}\|f_h\|_r^{r-1}-C_2(nh\ln n)^{-\frac{r}{2}}\right), 
\ee
where in the last line we have used \eqref{eqn:lr_odd_phih_bias_var} again. Choosing $C^*$ large enough so that
\be
\frac{\gamma_n(h)}{3}\|f_h\|_r^{r-1}-C_2(nh\ln n)^{-\frac{r}{2}} \ge \frac{\gamma_n(h)}{6}\|f_h\|_r^{r-1}, 
\ee 
which is possible due to the assumption $\|f_h\|_r\ge \frac{C_1}{\sqrt{nh\ln n}}$, and noting that
\begin{align*}
\|f\|_r \le \|f_h\|_r + \|f-f_h\|_r \le \frac{C_1}{\sqrt{nh\ln n}} + C_1h^s \le 2\|f_h\|_r, 
\end{align*}
Corollary \ref{cor:lr_odd_phih_moments} yields that for integer $k\ge 2$, 
\be
\ & \P\left(|T_h-\|f_h\|_r|>\frac{\gamma_n(h)}{3}\right)\\
&\le \left(\frac{6}{\gamma_n(h)\|f_h\|_r^{r-1}}\right)^k\cdot \bE \left|\int_{0}^1 T_h(x)dx-\E\left(\int_{0}^1 T_h(x)dx\right)\right|^k \\
&\le \frac{C_6}{\gamma_n(h)^k \|f_h\|_r^{k(r-1)}}\cdot n^{2k\epsilon}\left(\lambda_h^{kr}h^{\frac{k}{2}} + \lambda_h^k h^{\frac{k}{2r}} \|f_h\|^{k(r-1)}\right) \\
&\le C_7n^{2k\epsilon}(\ln n)^{\frac{k}{2}} \cdot \left( h^{\frac{k}{2}}\left(\frac{\lambda_h}{\|f_h\|_r}\right)^{k(r-1)} +  h^{\frac{k}{2r}}\right) \\
&\le C_8n^{2k\epsilon}(\ln n)^{\frac{kr}{2}}\cdot h^{\frac{k}{2r}}, \label{eqn:lr_odd_lepski_fhlarge_final}
\ee
where $C_6, C_7, C_8$ are constants depending on $(r,k,c_1,\sigma,K_M,C_1,C^*)$, and the last step follows from our assumption that $\|f_h\|_r \ge \frac{C_1}{\sqrt{nh\ln n}}$. 

Combining \eqref{eqn:lr_odd_lepski_fhsmall_final}, \eqref{eqn:lr_odd_lepski_error_decomp_fhlarge} and \eqref{eqn:lr_odd_lepski_fhlarge_final} we get that for any $h\leq h^*(s)$, $r\leq p$, integer $k\geq2$, 
\be 
\sup_{f\in \cB_{p,\infty}^{s}(L)}\P\left(|T_h-T_{h^*(s)}|>\gamma_n(h)\right) \le \max\{C_5,C_8\}4^{kr}n^{2k\epsilon}(\ln n)^{\frac{kr}{2}}\cdot (h^{\frac{k}{2r}} + h^*(s)^{\frac{k}{2r}}), \label{eqn:lr_odd_lepski_errorprob_final}
\ee
given that the constant $C^*$ is large enough. Note that $h\le h^*(s)\le (n\ln n)^{-\frac{1}{2s_{\max} + 1}}$, the inequality \eqref{eqn:lr_odd_lepski_errorprob_final} yields that as long as $4r\epsilon(2s_{\max}+1)<1$, choosing $k$ large enough will give
\be 
\sup_{f\in \cB_{p,\infty}^{s}(L)}\P\left(|T_h-T_{h^*(s)}|>\gamma_n(h)\right)&\leq \frac{C_9}{n^3}.\label{eqn:lr_odd_lepski_errorprob_proved}
\ee

Plugging the previous tail bound into \eqref{eqn:lr_even_lepski_term_II_prob}, we arrive at
\be
\text{II} \le \frac{C_9L^2|\mathcal{H}|}{n^3} \le \frac{C_9L^2 h_{\min}^{-1}}{n^3} = C_9L^2 \cdot \frac{\ln n}{n^2}\le C_{10}(n\ln n)^{-\frac{2s}{2s+1}}, 
\ee
and thereby complete the proof. \qed
\section{Technical Lemmas}\label{section:technical_lemmas}
In this section we collect some necessary technical lemmas necessary to prove the main results of this paper. We begin with a collection of lemmas available in literature which will serve as necessary tools to prove the other technical lemmas involved in the arguments laid down in Section \ref{section:main_proofs}.

\begin{lemma}\cite{cai2011testing}\label{lemma:hermite}
	For Hermite polynomial $H_k(x)$ of order $k$, if $X\sim\cN(\mu,1)$, we have
	\be
	\bE[H_k(X)] = \mu^k.
	\ee
	Moreover, if $|\mu|\le M$ and $k\le M^2$, we have
	\be
	\bE[H_k^2(X)] \le (2M^2)^k.
	\ee
\end{lemma}

\begin{lemma}\cite{bernstein1912ordre,varga1987conjecture}\label{lemma:approx_error}
	For any $r>0$, the best polynomial approximation error of $|x|^r$ on $[-1,1]$ satisfies
	\be
	\inf_{Q\in\mathsf{Poly}_n} \sup_{x\in [-1,1]} |Q(x)-|x|^r|\le \frac{\beta_r}{n^r}, 
	\ee
	where $\beta_r>0$ is a universal constant depending on $r$ only. Moreover, for $n$ large enough, we can choose $\beta_1$ to be the Bernstein constant $\beta_*=0.280169499$.
\end{lemma}

\begin{lemma}\cite{Boucheron--Lugosi--Massart2013}\label{lemma:gauss_tail}
	For $X\sim \cN(\mu,\sigma^2)$, we have
	\be
	\bP(|X-\mu|\ge t\sigma) \le \exp(-\frac{t^2}{2}). 
	\ee	
\end{lemma}

\begin{lemma}\label{lemma:gauss_monomials}
	Let $X\sim\mathcal{N}(\mu,\sigma^2)$ with $\mu>c\sigma\sqrt{\ln n}$, where $c>0, n\ge 2$. Then for any $\alpha\in\bR$ and any integer $k\ge 2$, we have
	\be
	&\bE\left|\mathbbm{1}(X\ge \frac{c\sigma\sqrt{\ln n}}{2})\cdot X^\alpha\right| \le D_\alpha\mu^\alpha,\\
	&\bE\left|\mathbbm{1}(X\ge \frac{c\sigma\sqrt{\ln n}}{2})\cdot X^\alpha - \bE \left(\mathbbm{1}(X\ge \frac{c\sigma\sqrt{\ln n}}{2})\cdot X^\alpha \right)\right|^k \le D_{\alpha,k}\sigma^k\mu^{(\alpha-1)k},\nonumber
	\ee
	where $D_\alpha, D_{\alpha,k}>0$ are universal constants depending on $\alpha,k$ and $c$ only.
\end{lemma}
\begin{proof}
	Throughout the proof we use the asymptotic notation $\lesssim$ to denote universal constants depending only on $(\alpha,c,k)$. 
	
	For the first inequality, define
	\be
	E_1 &\triangleq \{X: \frac{c\sigma\sqrt{\ln n}}{2} \le X<\frac{\mu}{2}\},\\
	E_2 &\triangleq \{X: \frac{\mu}{2}\le X\le 2\mu\}, \quad \text{and}\\
	E_3 &\triangleq \{X: X>2\mu\}.
	\ee
	By Lemma \ref{lemma:gauss_tail} and the triangle inequality, we have
	\be
	\ &\bE \left|\mathbbm{1}(X\ge \frac{c\sigma\sqrt{\ln n}}{2})\cdot X^\alpha\right| \le \sum_{i=1}^3\bE |\mathbbm{1}(E_i)\cdot X^\alpha|\\
	&\le \max\{ ( \frac{c\sigma\sqrt{\ln n}}{2})^\alpha, (\frac{\mu}{2})^\alpha \}\cdot \bP(E_1) + \max\{ (\frac{\mu}{2})^\alpha,(2\mu)^\alpha \} + \bE|X^\alpha \mathbbm{1}(X\ge 2\mu)|\\
	&\lesssim ((\sigma\sqrt{\ln n})^\alpha + \mu^\alpha) \cdot \exp(-\frac{\mu^2}{8\sigma^2}) + \mu^\alpha + \bE|X^\alpha \mathbbm{1}(X\ge 2\mu)|. 
	\ee
	For the last term, note that $X^\alpha \le (1+2^\alpha)(X-\mu)^\alpha$ holds for any $\alpha\in \bR$ when $X\ge 2\mu$, and
	\begin{align*}
	\bE[(X-\mu)^{2\alpha} \mathbbm{1}(X-\mu\ge c\sigma)] \lesssim \sigma^{2\alpha}
	\end{align*}
	by the scaling property of the Gaussian random variable $X-\mu\sim\mathcal{N}(0,\sigma^2)$. Hence, by Lemma \ref{lemma:gauss_tail} and Cauchy--Schwartz, 
	\begin{align*}
	\bE|X^\alpha \mathbbm{1}(X\ge 2\mu)| &\lesssim \bE|(X-\mu)^\alpha \mathbbm{1}(X\ge 2\mu)|\\
	&\le \bE^{1/2}|X^{2\alpha}\mathbbm{1}(X\ge 2\mu)| \cdot \bP^{1/2}(X\ge 2\mu) \\
	&\le \bE^{1/2}|(X-\mu)^{2\alpha} \mathbbm{1}(X-\mu\ge c\sigma)| \exp(-\frac{\mu^2}{4\sigma^2}) \lesssim \sigma^\alpha \exp(-\frac{\mu^2}{4\sigma^2}). 
	\end{align*}
	Combining the previous inequalities, we arrive at
	\be
	\bE \left|\mathbbm{1}(X\ge \frac{c\sigma\sqrt{\ln n}}{2})\cdot X^\alpha\right|  \lesssim (\sigma^\alpha + (\sigma\sqrt{\ln n})^\alpha + \mu^\alpha) \cdot \exp(-\frac{\mu^2}{8\sigma^2}) + \mu^\alpha. \label{eq.alpha_noncentral}
	\ee
	If $\alpha\ge 0$, the desired inequality follows from $\sigma^\alpha\lesssim (\sigma\sqrt{\ln n})^\alpha \lesssim \mu^\alpha$ and $\exp(-\frac{\mu^2}{8\sigma^2})\\ \le 1$. For $\alpha<0$, the facts 
	\be
	(\sigma\sqrt{\ln n})^\alpha \lesssim \sigma^\alpha, \quad \exp(-\frac{\mu^2}{8\sigma^2})\le 1, \quad \frac{\sigma^\alpha}{\mu^\alpha}\exp(-\frac{\mu^2}{8\sigma^2}) \le \max_{t\ge c\sqrt{\ln n}} t^{-\alpha}e^{-t^2/8} \lesssim 1
	\ee
	complete the proof of the desired inequality. 
	
	As for the second inequality, we first show the following chain of inequalities: 
	\be
	\ &\bE \left|\mathbbm{1}(X\ge \frac{c\sigma\sqrt{\ln n}}{2})\cdot X^\alpha - \mu^\alpha\right|^k\\
	&\lesssim \bE \left|\mathbbm{1}(X\ge \frac{c\sigma\sqrt{\ln n}}{2})\cdot (X^\alpha - \mu^\alpha)\right|^k + \mu^{k\alpha}\bP(X<\frac{c\sigma\sqrt{\ln n}}{2}) \label{eq.triangle}\\
	&\lesssim \sum_{i=1}^3 \bE[\mathbbm{1}(E_i)\cdot |X^\alpha-\mu^\alpha|^k] + \mu^{k\alpha}\exp(-\frac{\mu^2}{8\sigma^2})\label{eq.tail} \\\label{eq.pf_decay}
	&\lesssim \bE[\mathbbm{1}(E_2)\cdot |X^\alpha-\mu^\alpha|^k] + (\mu^{k\alpha} + (\sigma\sqrt{\ln n})^{k\alpha})\exp(-\frac{\mu^2}{8\sigma^2})\\\label{eq.pf_exp}
	&\lesssim \bE[\mathbbm{1}(E_2)\cdot |X^\alpha-\mu^\alpha|^k] + \sigma^k\mu^{(\alpha-1)k}\\
	\label{eq.pf_taylor}
	&\lesssim \sup_{\xi\in [\mu/2, 2\mu]}|\xi|^{k(\alpha-1)} \bE[\mathbbm{1}(E_2)\cdot |X-\mu|^k]  + \sigma^k\mu^{(\alpha-1)k} \\
	&\lesssim \sigma^k\mu^{(\alpha-1)k}\label{eq.final}. 
	\ee
	We elaborate on the inequalities \eqref{eq.triangle} to \eqref{eq.final} here: 
	\begin{enumerate}
		\item Inequality \eqref{eq.triangle} follows from the triangle inequality $|a+b|^k\le 2^{k-1}(|a|^k + |b|^k)$; 
		\item Inequality \eqref{eq.tail} follows from Lemma \ref{lemma:gauss_tail} and $c\sigma\sqrt{\ln n} \le \mu$; 
		\item Inequality \eqref{eq.pf_decay} follows from dealing with events $E_1$ and $E_3$ separately. For all $\alpha\in \bR$, the condition $c\sigma\sqrt{\ln n}/2\le X\le \mu/2$ implies that $|X^\alpha-\mu^\alpha|\lesssim (\sigma\sqrt{\ln n})^\alpha + \mu^\alpha$. Hence, by Lemma \ref{lemma:gauss_tail}, 
		\begin{align*}
		\E[\mathbbm{1}(E_1)\cdot |X^\alpha - \mu^\alpha|^k] &\lesssim \bP(E_1)\cdot ((\sigma\sqrt{\ln n})^\alpha + \mu^\alpha)^k \\
		&\lesssim (\mu^{k\alpha} + (\sigma\sqrt{\ln n})^{k\alpha})\exp(-\frac{\mu^2}{8\sigma^2}). 
		\end{align*}
		As for the event $E_3$, if $\alpha\le 0$ we have $|X^\alpha-\mu^\alpha|\le \mu^\alpha$ when $X\ge 2\mu$, and Lemma \ref{lemma:gauss_tail} gives
		\begin{align*}
		\E[\mathbbm{1}(E_3)\cdot |X^\alpha - \mu^\alpha|^k] \le \mu^{k\alpha}\P(E_3)\le \mu^{k\alpha}\exp(-\frac{\mu^2}{2\sigma^2}). 
		\end{align*}
		If $\alpha>0$, the condition $X\ge 2\mu$ ensures that
		\begin{align*}
		|X^\alpha - \mu^\alpha| \le X^\alpha \le 2^\alpha(X-\mu)^\alpha,
		\end{align*}
		and thus Lemma \ref{lemma:gauss_tail} with Cauchy--Schwartz yields
		\begin{align*}
		\bE [\mathbbm{1}(E_3)|X^\alpha - \mu^\alpha|^k] &\lesssim \bE[\mathbbm{1}(E_3)(X-\mu)^{k\alpha}] \\
		&\le \bP^{1/2}(E_3)\cdot \bE^{1/2}[(X-\mu)^{2k\alpha}] \\
		&\lesssim \sigma^{k\alpha}\exp(-\frac{\mu^2}{4\sigma^2}) \lesssim \mu^{k\alpha}\exp(-\frac{\mu^2}{4\sigma^2}),
		\end{align*}
		where the last step is due to the assumption $\sigma\lesssim \mu$; 
		\item If $\alpha\ge 0$, inequality \eqref{eq.pf_exp} follows from $(\sigma\sqrt{\ln n})^{k\alpha}\lesssim \mu^{k\alpha}$ and
		\be
		\frac{\mu^k}{\sigma^k}\cdot\exp(-\frac{\mu^2}{8\sigma^2}) \le \max_{t\ge c\sqrt{\ln n}} t^ke^{-t^2/8} \lesssim 1
		\ee
		for any integer $k\ge 2$. If $\alpha<0$, the desired inequality follows from 
		\be \mu^{k\alpha} + (\sigma\sqrt{\ln n})^{k\alpha} \lesssim \sigma^{k\alpha} \ee
		and
		\be
		\frac{\mu^{k(1-\alpha)}}{\sigma^{k(1-\alpha)}}\cdot\exp(-\frac{\mu^2}{8\sigma^2}) \le \max_{t\ge c\sqrt{\ln n}} t^{k(1-\alpha)}e^{-t^2/8} \lesssim 1
		\ee
		for any integer $k\ge 2$; 
		\item Inequality \eqref{eq.pf_taylor} follows from $X^\alpha - \mu^\alpha = \alpha\xi^{\alpha-1}(X-\mu)$ with some $\xi$ lying between $X$ and $\mu$, as well as the definition of $E_2$; 
		\item Inequality \eqref{eq.final} follows from $\bE|X-\mu|^k \lesssim \sigma^k$ for $X\sim \mathcal{N}(\mu,\sigma^2)$.
	\end{enumerate}
	
	Now observe that for any random variable $Y$ and integer $k\ge 2$, by the triangle inequality and Jensen's inequality, the following inequality
	\be
	\bE|Y-\bE Y|^k &\le 2^{k-1} \left(\bE|Y-c|^k + |\bE Y-c|^k\right) \\
	&\le 2^{k-1} \left(\bE|Y-c|^k + \bE|Y-c|^k\right) \\
	&= 2^k \bE|Y-c|^k
	\ee
	holds for any $c\in \mathbb{R}$. Hence, a combination of the previous two inequalities concludes the proof of the upper bound on $k$-th central moment.
\end{proof}
\begin{lemma}\cite[Thm. E]{Qazi--Rahman2007}\label{lemma:chebyshev}
	Let $p_n(x) = \sum_{\nu=0}^n a_\nu x^\nu$ be a polynomial of degree at most $n$ such that $|p_n(x)|\leq 1$ for $x\in [-1,1]$. Then, $|a_{n-2\mu}|$ is bounded above by the modulus of the corresponding coefficient of $T_n$ for $\mu = 0,1,\ldots,\lfloor n/2 \rfloor$, and $|a_{n-1-2\mu}|$ is bounded above by the modulus of the corresponding coefficient of $T_{n-1}$ for $\mu = 0,1,\ldots,\lfloor (n-1)/2 \rfloor$. Here $T_n(x)$ is the $n$-th Chebyshev polynomials of the first kind. 
\end{lemma}

\begin{lemma}\label{lemma:norm}
	Let $f$ and $f_h$ be defined in Section \ref{section:main_results}, and let $r\ge 1$ and $k>1$ be integers. There exists some universal constant $c$ depending only on $r,k$ and the kernel $K$ such that
	\be
	\|f_h\|_{k(r-1)}^{k(r-1)}\le ch^{(k-1)(-1+1/r)}\|f\|_r^{(k-1)(r-1)}\|f_h\|_r^{r-1}.
	\ee
\end{lemma}
\begin{proof}
	The proof follows from the proof of \cite[Lemma 4.4]{lepski1999estimation} and is hence omitted.
\end{proof}

\begin{lemma}\label{lemma:cailow1}
	Suppose $\mathbbm{1}(A)$ is an indicator random variable independent of $X$ and $Y$ and let $Z=X \mathbbm{1}(A) + Y \mathbbm{1}(A^c)$. Then
	\be
	\mathsf{Var}(Z) &= \mathsf{Var}(X) \bP(A) + \mathsf{Var}(Y) \bP(A^c)  + (\bE X - \bE Y)^2 \bP(A) \bP(A^c).
	\ee
	In general for any integer $k\geq 2$,
	\be
	\E|Z-\bE Z|^k &\le 2^{k-1}\left(\E|X-\bE X|^k\P(A) + \E|Y-\E Y|^k\P(A^c)+|\E X-\bE Y|^k\P(A)\P(A^c)\right).
	\ee
\end{lemma}
\begin{proof}
	The identity for $\var(Z)$ follows from \cite[Lemma 4]{cai2011testing}. For general $k$, by taking the expectation with respect to $\mathbbm{1}(A)$ first, we have
	\be
	\E|Z-\bE Z|^k &= \E |X\mathbbm{1}(A)+ Y\mathbbm{1}(A^c) - \E(X\mathbbm{1}(A)+ Y\mathbbm{1}(A^c))|^k \\
	&= \P(A)\E|X-\E X\cdot \P(A)-\E Y\cdot \P(A^c)|^k \\
	&\qquad + \P(A^c)\E|Y-\E X\cdot \P(A)-\E Y\cdot \P(A^c)|^k  .
	\ee
	By triangle inequality,
	\be
	\E|X-\E X\cdot \P(A)-\E Y\cdot \P(A^c)|^k &= \E|(X-\E X)+(\E X-\E Y)\P(A^c)|^k \\
	&\le 2^{k-1}\left(\E|X-\E X|^k + |\E X-\E Y|^k\P(A^c)^k\right), 
	\ee
	and thus
	\be
	\bE|Z-\bE Z|^k &\le \bP(A)\cdot 2^{k-1}(\bE|X-\bE X|^k + \bP(A^c)^k|\bE X-\bE Y|^k) \\
	&\qquad + \bP(A^c)\cdot 2^{k-1}(\bE|Y-\bE Y|^k + \bP(A)^k|\bE X-\bE Y|^k)   \\
	&\le 2^{k-1}\left(\bE|X-\bE X|^k\bP(A) + \bE|Y-\bE Y|^k\bP(A^c) + |\bE X-\bE Y|^k\bP(A)\bP(A^c)\right) ,
	\ee
	where the last step uses $\bP(A)^{k-1}+\bP(A^c)^{k-1}\le \bP(A)+\bP(A^c)=1$ for $k\ge 2$.
\end{proof}
%
%

\begin{lemma}\cite{Tsybakov2008}\label{lemma:TV_chi}
	The total variation distance and the chi-squared distance are related via the following inequality:
	\be
	V(P,Q) \le 1-\frac{1}{2}\exp(-\chi^2(P,Q)).
	\ee
\end{lemma}

\begin{lemma}\cite[Theorem C.2]{hardle2012wavelets}\label{lemma:rosenthal}
	Let $q\geq 2$ and let $X_1,\ldots,X_n$ be independent random variables such that $\E(X_i)=0$ and $\E|X_i|^q<\infty$.
	\be 
	\E\left(\left|\sum_{i=1}^nX_i\right|^q\right)&\leq C(q)\left[\sum_{i=1}^n \E|X_i|^q+\left(\sum_{i=1}^n \E|X_i|^2\right)^{q/2}\right].
	\ee
\end{lemma}

Finally, Lemma \ref{lmm:Kequiv} presents the equivalence between Peetre's $K$-functional and modulus of smoothness on $\bR$. For $p\in [1,\infty]$ and $r\in\mathbb{N}$, the Peetre's $K$-functional for $f$ defined on $\bR$ is defined as
\begin{align*}
K_r(f,t^r)_p \triangleq \inf_{g} \|f-g\|_p + t^r\|g^{(r)}\|_p, 
\end{align*}
where the infimum is taken over all functions $g$ defined on $\bR$ such that the derivative $g^{(r-1)}$ is locally absolutely continuous. Also recall the definition of the modulus of smoothness $\omega_r(f,t)_p$ in \eqref{eq.moduli_of_smoothness}.
\begin{lemma}
	\label{lmm:Kequiv}	
	For any $p\in [1,\infty]$ and $ r\in\mathbb{N}$, there exist universal constants $M=M(r,p)$ and $t_0=t_0(r,p)$ such that for any $0<t<t_0$ and $f$ defined on $\bR$,
	\begin{align}
	M^{-1}K_r(f,t^r)_p \le \omega_r(f,t)_p \le MK_r(f,t^r)_p.
	\label{eq:Kequiv}
	\end{align}
	Furthermore,
	\begin{equation}
	\omega_r(f,t)_p  \le M t^r \|f^{(r)}\|_p, \quad 0<t<t_0, 
	\label{eq:Kequiv1}
	\end{equation}
	and
	\begin{equation}
	\omega_r(f,t)_p  \le 2^r \|f\|_p, \quad t>0.
	\label{eq:Kequiv2}
	\end{equation}	
\end{lemma}
\begin{proof}
	The first inequality \eqref{eq:Kequiv} is due to \cite[Chapter 6, Theorem 2.4]{Devore--Lorentz1993}. For the other inequalities, \eqref{eq:Kequiv1} follows from \eqref{eq:Kequiv} by choosing $g=f$, and 
	\eqref{eq:Kequiv2} follows from the definition
	\eqref{eq.moduli_of_smoothness} and the triangle inequality.
\end{proof}

\subsection{Proof of Lemma \ref{lemma:lone_xi_bias_var}}
The proof of Lemma \ref{lemma:lone_xi_bias_var} follows in turn from sequence of lemmas. We first consider the case where $|f_h(x)|$ is small for which the next lemma is crucial. 
\begin{lemma}\label{lemma:small_regime}
	Let $|\mu|\le 2c_1\lambda_h\sqrt{\ln n}$, and $X\sim\cN(\mu,\lambda_h^2)$. Then for $c_2\ln n\ge 1, 4c_1^2\ge c_2$, the bias and variance of $P(X)$ in estimating $|\mu|$ can be upper bounded as
	\be
	|\bE P(X) - |\mu||&\le \frac{2c_1\beta_1}{c_2}\cdot \frac{\lambda_h}{\sqrt{\ln n}},\\
	\var(P(X)) &\le 2^{7c_2\ln n+4}c_1^2c_2^2\lambda_h^2(\ln n)^3,
	\ee
	where the constant $\beta_1$ appears in Lemma \ref{lemma:approx_error}.
\end{lemma}
\begin{proof}
	By Lemma \ref{lemma:hermite} we know that
	\be
	\bE P(X) = \sum_{k=0}^K g_{K,k}(2c_1\lambda_h\sqrt{\ln n})^{1-k}\cdot \mu^k.
	\ee
	By Lemma \ref{lemma:approx_error}, we have
	\be
	\sup_{x\in [-1,1]} \left|\sum_{k=0}^K g_{K,k}x^k - |x|\right| \le \frac{\beta_1}{K}.
	\ee
	By a variable substitution $x\mapsto \frac{\mu}{2c_1\lambda_h\sqrt{\ln n}}$, we obtain
	\be
	\sup_{|\mu|\le 2c_1\lambda_h\sqrt{\ln n}} \left|\sum_{k=0}^K g_{K,k}(2c_1\lambda_h\sqrt{\ln n})^{1-k}\cdot \mu^k - |\mu|\right| \le \frac{\beta_1}{K}\cdot 2c_1\lambda_h\sqrt{\ln n}.
	\ee
	Hence, the bias of $P(X)$ is upper bounded by
	\be
	|\bE P(X) - |\mu|| \le \frac{\beta_1}{K}\cdot 2c_1\lambda_h\sqrt{\ln n} = \frac{2c_1\beta_1}{c_2}\cdot \frac{\lambda_h}{\sqrt{\ln n}}, 
	\ee
	as desired. 
	
	As for the variance, first Lemma \ref{lemma:chebyshev} tells us
	\be
	|g_{K,k}| \le 2^{3K}, \qquad k=0,1,\cdots,K.
	\ee
	Hence, with the help of Lemma \ref{lemma:hermite}, we know that
	\be
	\var(P(X)) &\le \bE [P(X)^2]\\
	&\le (K+1)\sum_{k=0}^K |g_{K,k}|^2(2c_1\lambda_h\sqrt{\ln n})^{2(1-k)}\cdot \lambda_h^{2k}\bE[H_k(\frac{X}{\lambda_h})^2]\\
	&\le 2^{6K+1}K\sum_{k=0}^K (2c_1\lambda_h\sqrt{\ln n})^{2(1-k)}\cdot \lambda_h^{2k}[2(2c_1\sqrt{\ln n})^2]^k\\
	&\le 2^{7K+3}K\sum_{k=0}^K c_1^2\lambda_h^2\ln n\\
	&\le 2^{7K+4}c_1^2K^2\lambda_h^2\ln n, 
	\ee
	where we have used the fact that $k\le K\le (2c_1\sqrt{\ln n})^2$.
\end{proof}

The next lemma is useful to analyze the plug-in approach where $f_h(x)$ is large. The key observation in this regime is that the plug-in approach is almost unbiased due to the measure concentration property of Gaussian distribution.
\begin{lemma}\label{lemma:large_regime}
	Let $|\mu|\ge \frac{c_1}{2}\lambda_h\sqrt{\ln n}$, $X\sim\cN(\mu,\lambda_h^2)$, and $k\ge 2$ be any integer. Then
	\be
	|\bE|X|-|\mu||&\le \frac{4\lambda_h}{c_1\sqrt{\ln n}}\cdot n^{-c_1^2/8}, \\
	\bE||X|-\bE|X||^k &\le C(c_1,k)\lambda_h^{k},
	\ee
	where $C(c_1,k)$ is a universal constant depending on $k$ and $c_1$ only. In particular, when $k=2$, we have
	\be
	\var(|X|) \le C(c_1,2)\lambda_h^2.
	\ee
\end{lemma}
\begin{proof}
	By symmetry we can assume that $\mu\ge \frac{c_1}{2}\lambda_h\sqrt{\ln n}$, then the bias can be written as
	\be
	|\bE|X| - \mu| = |\bE|X|-\bE X| = 2\bE|X|\mathbbm{1}(X\le 0).
	\ee
	With the help of the Gaussian tail bound (Lemma \ref{lemma:gauss_tail}), we have
	\be
	|\bE|X| - \mu| &= 2\int_{-\infty}^{0} \bP(X\le t)dt \\
	&= 2\int_{-\infty}^{0} \bP(\frac{X-\mu}{\lambda_h}\le \frac{t-\mu}{\lambda_h})dt \\
	&\le 2\int_{-\infty}^{0} \exp(-\frac{(t-\mu)^2}{2\lambda_h^2})dt\\
	&\le \frac{2\lambda_h^2}{\mu}\int_{-\infty}^{0} \frac{\mu-t}{\lambda_h^2}\exp(-\frac{(t-\mu)^2}{2\lambda_h^2})dt\\
	&= \frac{2\lambda_h^2}{\mu}\exp(-\frac{\mu^2}{2\lambda_h^2})\le \frac{4\lambda_h}{c_1\sqrt{\ln n}}\cdot n^{-c_1^2/8},
	\ee
	which completes the proof of the bias bound. 
	
	As for the $k$-th central moment, we have
	\be
	\ & \bE ||X|-\bE |X||^k \\
	&\le 3^{k-1} \left(\bE||X|-X|^k + \bE|X-\mu|^k + |\bE|X|-\mu|^k\right) \\
	&\le 3^{k-1}\left(2^k\bE |X|^k\mathbbm{1}(X\le 0) + \lambda_h^k (k-1)!! + \left(\frac{4\lambda_h}{c_1\sqrt{\ln n}}\cdot n^{-c_1^2/8}\right)^k\right)
	\ee
	where we have used the previous bias bound for the last term. Using the same technique as before, 
	\be
	\bE |X|^k\mathbbm{1}(X\le 0) &= k\int_{-\infty}^{0} |t|^{k-1}\bP(X\le t)dt \\
	&\le k\int_{-\infty}^{0} |t|^{k-1}\exp\left(-\frac{(t-\mu)^2}{2\lambda_h^2}\right)dt\\
	&\le k\lambda_h^k\int_{-\infty}^{0} \frac{1}{\lambda_h}\left(\frac{\mu-t}{\lambda_h}\right)^{k-1}\exp\left(-\frac{(t-\mu)^2}{2\lambda_h^2}\right)dt\\
	&= k\lambda_h^k \cdot \sqrt{2\pi}\bE\left(Z^{k-1}\mathbbm{1}\left(Z\ge \frac{c_1}{2}\sqrt{\ln n}\right)\right) \\
	&\leq C\lambda_h^kn^{-c_1^2/16},
	\ee
	where $Z=\frac{\mu-X}{\lambda_h}\sim \mathcal{N}(0,1)$ is a standard normal random variable, the constant $C$ depends on $c_1$ and $k$, and the last inequality is due to Cauchy--Schwartz
	\be
	\bE\left(Z^{k-1}\mathbbm{1}\left(Z\ge \frac{c_1}{2}\sqrt{\ln n}\right)\right) \le \bE^{1/2}(Z^{2k-2})\bP^{1/2}(Z\ge \frac{c_1}{2}\sqrt{\ln n})
	\ee
	and Lemma \ref{lemma:gauss_tail}. Hence the $k$-th central moment bound is proved.
\end{proof}
\begin{proof}[Proof of Lemma \ref{lemma:lone_xi_bias_var}]
	Throughout the proof, for any two sequences $a_n,b_n$, we will use the notation $a_n\lesssim b_n$ whenever $|\frac{a_n}{b_n}|$ is upper bounded by a universal constant which only depends on $c_1, c_2, k$. 
	
	Note that $X,Y$ are independent, Lemma \ref{lemma:cailow1} can be employed here to establish upper bounds on $k^{\mathrm{th}}$ central moments of $\xi(X,Y)$. We distinguish into three cases: 
	\begin{enumerate}
		\item Case I: $|\mu|\le \frac{c_1}{2}\lambda_h\sqrt{\ln n}$. By Lemma \ref{lemma:small_regime} and Markov's inequality,
		\be
		|\bE \tilde{P}(X) - |\mu|| &\le |\bE P(X)-|\mu|| + \bE|P(X)-\tilde{P}(X)| \\
		&= |\bE P(X)-|\mu|| + \bE|P(X)\mathbbm{1}(|P(X)|\ge n^{2\epsilon}\lambda_h)| \\
		&\lesssim \frac{\lambda_h}{\sqrt{\ln n}} + \frac{\bE|P(X)|^2}{n^{2\epsilon}\lambda_h} \lesssim  \frac{\lambda_h}{\sqrt{\ln n}} + \frac{\lambda_h}{n^{\epsilon}}. \label{eq.P_tilde}
		\ee
		
		Hence, 
		\be
		|\bE\xi(X,Y)-|\mu|| &\le |\bE \tilde{P}(X)-|\mu||  + (\bE|\tilde{P}(X)| + \bE|X|)\cdot \bP(|Y|\ge c_1\lambda_h\sqrt{\ln n})\\
		&\le |\bE \tilde{P}(X)-|\mu|| + (n^{2\epsilon}\lambda_h + |\mu| + \lambda_h)\cdot \bP\left(\left|\frac{Y-\mu}{\lambda_h}\right|\ge \frac{c_1}{2}\sqrt{\ln n}\right)\\
		&\lesssim \frac{\lambda_h}{\sqrt{\ln n}} + \frac{\lambda_h}{n^{\epsilon}} + (n^{2\epsilon}\lambda_h + \frac{c_1}{2}\lambda_h\sqrt{\ln n} + \lambda_h)\cdot n^{-c_1^2/8}
		\lesssim \frac{\lambda_h}{\sqrt{\ln n}},
		\ee
		where 
		\begin{enumerate}
			\item the first inequality follows from the triangle inequality; 
			\item the second inequality follows from $\bE|X|\le \mu+\bE|X-\mu|\le \mu+\lambda_h$;
			\item the third inequality follows from \eqref{eq.P_tilde}, the assumption $|\mu|\le \frac{c_1}{2}\lambda_h\sqrt{\ln n}$ and Lemma \ref{lemma:gauss_tail}; 
			\item the last inequality follows from $c_1>\sqrt{8k}$. 
		\end{enumerate} 
		
		As for the $k$-th central moment, Lemma \ref{lemma:cailow1} yields
		\be
		\ &\E|\xi(X,Y)-\E \xi(X,Y)|^k \\
		&\lesssim \E|\tilde{P}(X)-\E\tilde{P}(X)|^k + (\E||X|-\E |X||^k + |\E\tilde{P}(X) - \E |X||^k) \cdot \P(|Y|\ge c_1\lambda_h\sqrt{\ln n})\\
		&\lesssim \E|\tilde{P}(X)|^k + \left(\lambda_h^k + (n^{2\epsilon}\lambda_h+|\mu|+\lambda_h)^k\right)\cdot n^{-c_1^2/8}\lesssim (\lambda_hn^{2\epsilon})^k,
		\ee
		where the second inequality follows from the triangle inequality, Lemma \ref{lemma:large_regime} and $|\E\tilde{P}(X) - \E |X|| \le |\bE \tilde{P}(X)| + |\mu| + \bE|X-\mu| \le n^{2\epsilon}\lambda_h+|\mu|+\lambda_h$. 
		
		\item Case II: $\frac{c_1}{2}\lambda_h\sqrt{\ln n}<|\mu|<2c_1\lambda_h\sqrt{\ln n}$.
		For the bias bound, we employ the triangle inequality, \eqref{eq.P_tilde} and Lemma \ref{lemma:large_regime} to obtain
		\be
		|\bE\xi(X,Y)-|\mu|| &\le |\bE \tilde{P}(X) - |\mu|| + |\bE |X| - |\mu||\\
		&\lesssim \frac{\lambda_h}{\sqrt{\ln n}} + \frac{\lambda_h}{n^{\epsilon}} + \frac{\lambda_h}{\ln n}\cdot n^{-c_1^2/8} \lesssim \frac{\lambda_h}{\sqrt{\ln n}}. 
		\ee
		
		As for the $k$-th central moment, by Lemma \ref{lemma:cailow1} we have
		\be
		\ &\bE|\xi(X,Y)-\bE\xi(X,Y)|^k \\
		&\lesssim \bE|\tilde{P}(X)-\bE\tilde{P}(X)|^k + \bE||X|-\bE|X||^k + |\bE\tilde{P}(X) - \bE |X||^k\\
		&\lesssim \bE|\tilde{P}(X)|^k + \bE||X|-\bE|X||^k + |\bE\tilde{P}(X) - |\mu||^k + |\bE|X|-|\mu||^k\\
		&\lesssim (\lambda_hn^{2\epsilon})^k + \lambda_h^k + \left(\frac{\lambda_h}{\sqrt{\ln n}}+\frac{\lambda_h}{n^{\epsilon}}\right)^k + \left(\frac{\lambda_h}{\sqrt{\ln n}}\cdot n^{-c_1^2/8}\right)^k\lesssim (\lambda_hn^{2\epsilon})^k,
		\ee
		where the third inequality follows from Lemma \ref{lemma:large_regime} and \eqref{eq.P_tilde}. 
		
		\item Case III: $|\mu|\ge 2c_1\lambda_h\sqrt{\ln n}$. By Lemma \ref{lemma:large_regime} and Lemma \ref{lemma:gauss_tail}, the bias bound is given by
		\be
		|\bE\xi(X,Y)-|\mu||&\le |\bE|X|-|\mu|| + (\bE|X| + \bE|\tilde{P}(X)|)\cdot \bP(|Y|\le c_1\lambda_h\sqrt{\ln n})\\
		&\le |\bE|X|-|\mu|| + (\bE|X| + \bE|\tilde{P}(X)|)\cdot \bP\left(\left|\frac{Y-\mu}{\lambda_h}\right|\ge \frac{|\mu|}{2\lambda_h}\right)\\
		&\lesssim \frac{\lambda_h}{\sqrt{\ln n}}\cdot n^{-c_1^2/8} + (|\mu|+\lambda_h+\lambda_hn^{2\epsilon})\cdot \exp(-\frac{\mu^2}{8\lambda_h^2})\lesssim \frac{\lambda_h}{\sqrt{\ln n}},
		\ee
		where the last inequality follows from $c_1>\sqrt{8k}$ and
		$$
		\sup_{|\mu| \ge 2c_1\lambda_h\sqrt{\ln n}} |\mu| \exp(-\frac{\mu^2}{8\lambda_h^2}) \lesssim \frac{1}{n}.
		$$
		
		By Lemmas \ref{lemma:gauss_tail}, \ref{lemma:cailow1} and \ref{lemma:large_regime}, the $k$-th central moment can be upper bounded as
		\be
		\ &\bE|\xi(X,Y)-\bE\xi(X,Y)|^k \\
		&\lesssim \bE(|X|-\bE|X|)^k + (\bE|\tilde{P}(X)-\bE\tilde{P}(X)|^k + |\bE|X| - \bE\tilde{P}(X)|^k)\cdot \bP(|Y|\le c_1\lambda_h\sqrt{\ln n})\\
		&\lesssim \lambda_h^k + ((\lambda_hn^{2\epsilon})^k + (|\mu|+\lambda_h+\lambda_hn^{2\epsilon})^k)\cdot \exp(-\frac{\mu^2}{8\lambda_h^2})\lesssim \lambda_h^k,
		\ee
	\end{enumerate}
	where again the last step follows from taking supremum over $|\mu|\ge 2c_1\lambda_h\sqrt{\ln n}$.   
	
	Combining these three cases completes the proof of the lemma.
\end{proof}

\subsection{Proof of Lemma \ref{lemma:overall_lr_odd}} The proof of Lemma \ref{lemma:overall_lr_odd} follows in turn from sequence of lemmas. We first consider the case where $|f_h(x)|$ is small for which the next lemma is crucial.

\begin{lemma}\label{lemma:small_regime_lr}
	Let $|\mu|\le 2c_1\lambda_h\sqrt{\ln n}$, and $X\sim\cN(\mu,\lambda_h^2)$. Then for $c_2\ln n\ge 1, 4c_1^2\ge c_2$, the bias and variance of $P_r(X)$ in estimating $|\mu|^r$ can be upper bounded as
	\be
	|\bE P_r(X) - |\mu|^r|&\le \beta_r\cdot \left(\frac{2c_1\lambda_h}{c_2\sqrt{\ln n}}\right)^r,\\
	\var(P_r(X)) &\le 2^{7c_2\ln n+2}c_2^2(\ln n)^2\cdot (2c_1\lambda_h\sqrt{\ln n})^{2r},
	\ee
	where the constant $\beta_r$ appears in Lemma \ref{lemma:approx_error}.
\end{lemma}
\begin{proof}
	By Lemma \ref{lemma:hermite} we know that
	\be
	\bE P_r(X) = \sum_{k=0}^K g_{K,k}^{(r)}(2c_1\lambda_h\sqrt{\ln n})^{r-k}\cdot \mu^k.
	\ee
	By Lemma \ref{lemma:approx_error}, we have
	\be
	\sup_{x\in [-1,1]} \left|\sum_{k=0}^K g_{K,k}^{(r)}x^k - |x|^r\right| \le \frac{\beta_r}{K^r}.
	\ee
	By a variable substitution $x\mapsto \frac{\mu}{2c_1\lambda_h\sqrt{\ln n}}$, we obtain
	\be
	\sup_{|\mu|\le 2c_1\lambda_h\sqrt{\ln n}} \left|\sum_{k=0}^K g_{K,k}^{(r)}(2c_1\lambda_h\sqrt{\ln n})^{r-k}\cdot \mu^k - |\mu|^r\right| \le \beta_r\cdot \left(\frac{2c_1\lambda_h\sqrt{\ln n}}{K}\right)^r.
	\ee
	Hence, the bias of $P_r(X)$ is upper bounded by
	\be
	|\bE P_r(X) - |\mu|^r| \le \beta_r\cdot \left(\frac{2c_1\lambda_h\sqrt{\ln n}}{K}\right)^r = \beta_r\cdot \left(\frac{2c_1\lambda_h}{c_2\sqrt{\ln n}}\right)^r, 
	\ee
	as desired. 
	
	As for the variance, first Lemma \ref{lemma:chebyshev} tells us
	\be
	|g_{K,k}| \le 2^{3K}, \qquad k=0,1,\cdots,K.
	\ee
	Hence, with the help of Lemma \ref{lemma:hermite}, we know that
	\be
	\var(P_r(X)) &\le \bE [P_r(X)^2]\\
	&\le (K+1)\sum_{k=0}^K |g_{K,k}^{(r)}|^2(2c_1\lambda_h\sqrt{\ln n})^{2(r-k)}\cdot \lambda_h^{2k}\bE\left[H_k(\frac{X}{\lambda_h})^2\right]\\
	&\le 2^{6K+1}K\sum_{k=0}^K (2c_1\lambda_h\sqrt{\ln n})^{2(r-k)}\cdot \lambda_h^{2k}[2(2c_1\sqrt{\ln n})^2]^k\\
	&\le 2^{7K+1}K\sum_{k=0}^K (2c_1\lambda_h\sqrt{\ln n})^{2r}\\
	&\le 2^{7K+2}K^2(2c_1\lambda_h\sqrt{\ln n})^{2r}, 
	\ee
	where we have used the fact that $k\le K\le (2c_1\sqrt{\ln n})^2$.
\end{proof}

Next we analyze the ``smooth" regime where $|f_h(x)|$ is large. If $f_h(x)>0$, the Taylor expansion based estimator $S_{\lambda_h}(\tilde{f}_{h,1}(x),\tilde{f}_{h,2}(x))$ is analyzed in detail in the following lemma. The analysis of the estimator $S_{\lambda_h}(-\tilde{f}_{h,1}(x), -\tilde{f}_{h,2}(x))$ in the case $f_h(x)<0$ then follows by symmetry. Subsequently, we can take into account the sample splitting approach, and following the same approach as of the proof of Lemma \ref{lemma:lone_xi_bias_var} one can complete the proof of Lemma \ref{lemma:overall_lr_odd}. We omit the details.

\begin{lemma}\label{lemma:large_regime_lr}
	Let $\mu\ge \frac{c_1}{2}\lambda_h\sqrt{\ln n}$, $k\ge 2$ be any integer, and $X_1,X_2\sim\cN(\mu,\lambda_h^2)$ be independent. The bias and $k$-th central moment of $S_{\lambda_h}(X_1,X_2)$ in estimating $\mu^r$ can be upper bounded as
	\be
	|\bE S_{\lambda_h}(X_1,X_2)-\mu^r|&\le \lambda_h^r(\sqrt{\ln n})^{r-R-1}+ (\lambda_h\sqrt{\ln n})^r n^{-c_1^2/32}, \\
	\bE|S_{\lambda_h}(X_1,X_2) - \bE S_{\lambda_h}(X_1,X_2)|^k &\le C(r,k)\lambda_h^{k}\mu^{(r-1)k},
	\ee
	where $C(r,k)>0$ is a universal constant depending only on $r,k$ and $c_1$. In particular, for $k=2$ we have
	\be
	\var(S_{\lambda_h}(X_1,X_2)) \le C(r,2)\lambda_h^2\mu^{2r-2}.
	\ee
\end{lemma}
\begin{proof}
	Throughout the proof we use the notation $a_n\lesssim b_n$ to show that $|\frac{a_n}{b_n}|$ is upper bounded by a universal constant which only depends on $c_1, k$ and $r$. 
	
	First we analyze the bias. By Lemma \ref{lemma:hermite} and independence of $X_1$ and $X_2$, 
	\be
	\bE S_{\lambda_h}(X_1, X_2) &= \bE\left[\mathbbm{1}(X_1\ge \frac{c_1}{4}\lambda_h\sqrt{\ln n})\cdot \sum_{k=0}^R a_kX_1^{r-k}\sum_{j=0}^k \binom{k}{j}\bE_{X_2}\left(\lambda_h^jH_j(\frac{X_2}{\lambda_h})\right)(-X_1)^{k-j}\right] \\
	&= \bE\left[\mathbbm{1}(X_1\ge \frac{c_1}{4}\lambda_h\sqrt{\ln n})\cdot \sum_{k=0}^R a_kX_1^{r-k}\sum_{j=0}^k \binom{k}{j}\mu^j(-X_1)^{k-j}\right] \\
	& = \bE \left[\mathbbm{1}(X_1\ge \frac{c_1}{4}\lambda_h\sqrt{\ln n})\cdot \sum_{k=0}^R a_kX_1^{r-k}(\mu-X_1)^k\right], 
	\ee
	where $a_k\triangleq\frac{r(r-1)\cdots(r-k+1)}{k!}$ is the Taylor coefficient. Note that by the Taylor expansion with Lagrange remainder term, we have
	\begin{align*}
	\mu^r - \sum_{k=0}^R a_kX_1^{r-k}(\mu-X_1)^k = a_{R+1}\xi^{r-R-1}(\mu-X_1)^{R+1},
	\end{align*}
	for some $\xi$ lying between $X_1$ and $\mu$. In view of $\mu \ge \frac{c_1}{2}\lambda_h\sqrt{\ln n}$ and $X_1\ge \frac{c_1}{4}\lambda_h\sqrt{\ln n}$, we conclude that $\xi \ge \frac{c_1}{4}\lambda_h\sqrt{\ln n}$. Hence, the triangle inequality yields
	\be
	\ &|\bE S_{\lambda_h}(X_1,X_2) - \mu^r| \\
	&\le \mu^r \bP(X_1<\frac{c_1}{4}\lambda_h\sqrt{\ln n}) +  \bE \left|\mathbbm{1}(X_1\ge \frac{c_1}{4}\lambda_h\sqrt{\ln n})\cdot \left(\sum_{k=0}^R a_kX_1^{r-k}(\mu-X_1)^k-\mu^k\right)\right|\\
	&\le  \mu^r \bP(|X_1-\mu|\ge \frac{\mu}{2}) + \sup_{\xi\ge \frac{c_1}{4}\lambda_h\sqrt{\ln n}}\bE|\mathbbm{1}(X_1\ge \frac{c_1}{4}\lambda_h\sqrt{\ln n})\cdot a_{R+1}\xi^{r-R-1}(\mu-X_1)^{R+1}|\\
	&\le 2\mu^r\exp(-\frac{\mu^2}{8\lambda_h^2}) + \left(\frac{c_1}{4}\lambda_h\sqrt{\ln n}\right)^{r-R-1}|a_{R+1}|\cdot\bE|\mu-X_1|^{R+1}\\
	&\lesssim (\lambda_h\sqrt{\ln n})^r n^{-c_1^2/32} + \lambda_h^r(\sqrt{\ln n})^{r-R-1},
	\ee
	where we have used Lemma \ref{lemma:gauss_tail}, $r-R-1<0$ and
	\be
	\ & \max_{\mu \ge c_1 \lambda_h \sqrt{\ln n} / 2} \mu^r \exp(-\frac{\mu^2}{8\lambda_h^2}) \lesssim (\lambda_h\sqrt{\ln n})^r n^{-c_1^2/32}, \\
	& \bE|X_1-\mu|^k \le \sqrt{\bE|X_1-\mu|^{2k}} \lesssim \lambda_h^{k}, \qquad k=0,1,\cdots.
	\ee
	This proves the bias bound. 
	
	As for the $k$-th central moment, we may write
	\be
	S_{\lambda}(X_1, X_2) = \sum_{\ell=0}^R \sum_{j=0}^\ell b_{j,\ell}\zeta_j \eta_j,
	\ee
	where
	\be
	b_{j,\ell} &\triangleq \frac{r(r-1)\cdots(r-\ell+1)}{\ell!}\cdot (-1)^{\ell-j}\binom{\ell}{j}, \\
	\zeta_j &\triangleq X_1^{r-j}\mathbbm{1}(X_1\ge \frac{c_1}{4}\lambda_h\sqrt{\ln n}), \\
	\eta_j &\triangleq \lambda_h^j H_j(\frac{X_2}{\lambda_h}). 
	\ee
	By independence of $X_1$ and $X_2$, 
	\be
	S_{\lambda}(X_1, X_2) - \E[S_{\lambda}(X_1,X_2)] = \sum_{\ell=0}^R \sum_{j=0}^{\ell} b_{j,\ell}[(\zeta_j - \E[\zeta_j])\E[\eta_j] + \zeta_j(\eta_j - \E[\eta_j])]. 
	\ee
	
	Since the coefficients $b_{j,\ell}$ do not depend on $\mu$ or $\lambda_h$, by the triangle inequality it suffices to prove that for any $j=0,1,\cdots,R$, 
	\be
	A_j &\triangleq \E[|\zeta_j - \E[\zeta_j]|^k] \cdot |\E[\eta_j]|^k \lesssim \lambda_h^k\mu^{(r-1)k}, \label{eq.A_1}\\
	B_j &\triangleq \E[|\zeta_j|^k]\cdot \E[|\eta_j - \E[\eta_j]|^k ] \lesssim \lambda_h^k \mu^{(r-1)k}. \label{eq.A_2}
	\ee
	
	For the first term $A_j$, the second inequality of Lemma \ref{lemma:gauss_monomials} with $\sigma=\lambda_h, \alpha=r-j$ and $c=c_1/2$ gives
	$$
	\E[|\zeta_j - \E[\zeta_j]|^k] \lesssim \lambda_h^k\mu^{(r-j-1)k}. 
	$$
	Moreover, Lemma \ref{lemma:hermite} gives $\E[\eta_j] = \mu^j$, and therefore \eqref{eq.A_1} holds. 
	
	For the second term $B_j$, we may assume that $j\ge 1$ since $B_0=0$. For the $k$-th moment of $\zeta_j$, the first inequality of Lemma \ref{lemma:gauss_monomials} with $\sigma=\lambda_h, \alpha=r-j$ and $c=c_1/2$ gives $\E[|\zeta|^k]\lesssim \mu^{(r-j)k}$. To upper bound the $k$-th central moment of $\eta_j$, we express the $j$-th Hermite polynomial $H_j$ as $H_j(x) = \sum_{i=0}^j a_{i,j}x^i$, then 
	\be 
	\eta_j - \E[\eta_j] = \sum_{i=0}^j a_{i,j}\lambda_h^{j-i}(X_2^i - \E[X_2^i]) = \sum_{i=1}^j a_{i,j}\lambda_h^{j-i}(X_2^i - \E[X_2^i]). 
	\ee 
	Write $X_2=\mu+\lambda_h Z$, with $Z\sim \mathcal{N}(0,1)$. Then for $i\ge 1$, 
	\begin{align*}
	X_2^i - \bE X_2^i = (\mu+\lambda_h Z)^i - \bE (\mu+\lambda_h Z)^i = \sum_{i'=1}^{i} \binom{i}{i'}\mu^{i-i'} \lambda_h^{i'} (Z^{i'} - \bE[Z^{i'}]). 
	\end{align*}
	Since all moments of $Z$ are finite, the $i'$-th summand has $k$-th moment $\lesssim \mu^{(i-i')k}\lambda_h^{i'k}$, and the triangle inequality then yields
	\be
	\bE|X_2^i-\bE X_2^i|^k \lesssim \sum_{i'=1}^i \mu^{(i-i')k}\lambda_h^{i'k} \lesssim \lambda_h^{k}(\mu^{(i-1)k}+\lambda_h^{(i-1)k}) \lesssim \lambda_h^k \mu^{(i-1)k},
	\ee
	where the last inequality follows from $i\ge 1$ and the assumption $\lambda_h\lesssim \mu$. Therefore, by a triangle inequality again, 
	\be 
	\E[|\eta_j - \E[\eta_j]|^k] \lesssim \sum_{i=1}^j \lambda_h^{(j-i)k}\cdot \bE|X_2^i-\bE X_2^i|^k \lesssim \lambda_h^k \sum_{i=1}^j \lambda_h^{(j-i)k}\mu^{(i-1)k} \lesssim \lambda_h^k \mu^{(j-1)k}.
	\ee
	Finally, combining the previous upper bounds yields
	\be 
	B_j = \E[|\zeta_j|^k]\cdot \E[|\eta_j - \E[\eta_j]|^k ] \lesssim \mu^{(r-j)k}\cdot \lambda_h^k \mu^{(j-1)k} = \lambda_h^k \mu^{(r-1)k}, 
	\ee 
	which gives \eqref{eq.A_2}. Combining \eqref{eq.A_1} and \eqref{eq.A_2} completes the proof of the upper bound for the $k$-th central moment. 
\end{proof}

\subsection{Proof of Lemma \ref{lemma:lr_odd_phih_moments}}
For the simplicity of the proof notation, we shall assume that $J=1/h$ is an integer. The more general proof follows with obvious modifications by working with $\lfloor 1/h\rfloor$. We only provide the proof here for the case when $J$ is an even integer i.e. $J=2L$ for some $L\geq 1$. The proof for $J$ odd can be obtained similarly. 

In particular, we consider the partition of $[0,1]$ into $2L$ consecutive subintervals of length $h$ each and for $l=0,1,\ldots,2L-1$ denote the $l^{\mathrm{th}}$ subinterval by $I_l$ i.e. $I_l=[(l-1)h,lh)$ for $l=0,\ldots,2L-2$ and $I_{2L-1}=[(2L-1)h,2Lh]$. Let $I_1=\bigcup_{l=0}^{L-1}I_{2l}$ and $I_2=\bigcup_{l=0}^{L-1}I_{2l+1}$. Then
\be 
\ & \int_0^1 \left(T_h(x)-\E T_h(x)\right)dx\\
&=\int_{I_1} \left(T_h(x)-\E T_h(x)\right)dx+\int_{I_2} \left(T_h(x)-\E T_h(x)\right)dx\\
&=T_1+T_2.
\ee
Indeed,
\be 
\E\left|\int_0^1 \left(T_h(x)-\E T_h(x)\right)dx\right|^k&\leq 2^{k-1}\left(\E|T_1|^k+\E|T_2|^k\right).\label{eqn:lr_phih_moment}
\ee
We now provide control over $\E|T_1|^k$. The bound over $\E|T_2|^k$ is similar and combining them shall yield the desired proof of the lemma.
First note that
\be 
T_1=\sum_{l=0}^{L-1}\int_{2lh}^{(2l+1)h}(T_h(x)-\E T_h(x))dx=\sum_{l=0}^{L-1} \xi_{h,l},
\ee
where $\xi_{h,l}(x)=\int_{2lh}^{(2l+1)h}(T_h(x)-\E T_h(x))dx$ are independent and zero-mean random variables for $l=0,\ldots,L-1$. Therefore, by Rosenthal's Inequality (Lemma \ref{lemma:rosenthal}) we have that 
\be 
\E|T_1|^k=\E\left|\sum_{l=1}^{L-1} \xi_{h,l}\right|^k&\leq C(k)\left[\sum_{l=0}^{L-1}\E|\xi_{h,l}|^k+\left(\sum_{l=0}^{L-1} \E|\xi_{h,l}|^2\right)^{k/2}\right].\label{eqn:lr_even_partition_moments}
\ee
Now, by Jensen's Inequality on the interval of length $h$
\be 
\E|\xi_{h,l}|^k&=\E\left|\int_{2lh}^{(2l+1)h}(T_h(x)-\E(T_h(x)))dx\right|^k\\
&\leq h^k\frac{1}{h}\int_{2lh}^{(2l+1)h}\E|T_h(x)-\E(T_h(x))|^k dx \\
&\leq h^k\frac{1}{h}\int_{2lh}^{(2l+1)h}C_2n^{2k\epsilon}\left(\lambda_h^{kr} + \lambda_h^k|f_h(x)|^{k(r-1)}\right), \label{eqn:lr_even_partition_jensens}
\ee
where the last inequality follows from Lemma \ref{lemma:lone_xi_bias_var} (for $r=1$) and Lemma \ref{lemma:overall_lr_odd} (for $r>1$) with $C_2$ a constant depending on $c_1,c_2,\epsilon,\sigma,K_M,k$. Plugging in the bound \eqref{eqn:lr_even_partition_jensens} into \eqref{eqn:lr_even_partition_moments} and subsequently combining with \eqref{eqn:lr_phih_moment} completes the proof of Lemma \ref{lemma:lr_odd_phih_moments}. \qed

\subsection{Proof of Lemma \ref{lem.measure}}\label{subsec:measure_proof}
Let $S$ be the optimal value of the optimization program in \eqref{eq:Rstar}, then the target is to prove that $S = 2E_{q-1,K}(f;I)$. We first show that $S\le 2E_{q-1,K}(f;I)$. Let $P(x) = \sum_{k = -q+1}^K a_kx^k$ be the best approximating rational function attaining the approximation error $E_{q-1,K}(f;I)$, then for any feasible pair of probability measures $(\nu_0, \nu_1)$ in \eqref{eq:Rstar}, the triangle inequality gives
\begin{align*}
\int f(x)(\nu_1(dx) - \nu_0(dx)) &= \int (f(x) - P(x))(\nu_1(dx) - \nu_0(dx)) \\
&\le \int |f(x) - P(x)|(\nu_1(dx) + \nu_0(dx)) \\
&\le \int E_{q-1,K}(f;I)(\nu_1(dx) + \nu_0(dx)) \\
&= 2E_{q-1,K}(f;I). 
\end{align*}
Consequently, $S\le 2E_{q-1,K}(f;I)$. To show the other inequality, we construct the probability measures $\nu_0$ and $\nu_1$ explicitly. Since the interval $I$ does not contain zero, the functions $\{x^{-q+1}, x^{-q+2}, \cdots, x^K \}$ in $C(I)$ form a Chebyshev system \cite[Section 3.3, Example 2]{Devore--Lorentz1993}, and the Chebyshev alternation theorem \cite[Chapter 3, Theorem 5.1]{Devore--Lorentz1993} shows that there exist a rational function $P(x) = \sum_{k=-q+1}^K a_kx^k$ and points $x_0<x_1<\cdots<x_{q+K}$ in $I$ with $f(x_i) - P(x_i) = \varepsilon\cdot (-1)^iE_{q-1,K}(f;I)$ for all $i=0,1,\cdots,q+K$, and $\varepsilon\in \{\pm 1\}$. Construct the signed measure $\nu$ supported on $\{x_0,x_1,\cdots,x_{q+K}\}$ with
\begin{align*}
\nu(\{x_i\}) = c_0\cdot x_i^{q-1}\prod_{j\neq i} \frac{1}{x_i - x_j}, 
\end{align*}
where $c_0\in \bR$ is a scaling factor such that $|\nu|(\bR) = 1$ and $\nu(\{x_0\})$ has the same sign as $\varepsilon$. 
For each $k = -q+1,\cdots,K$, using the divided difference of $x\mapsto x^{k+q-1}$, we have
\begin{align*}
\int x^k\nu(dx) = c_0\cdot \sum_{i=0}^{q+K} x_i^{k+q-1}\prod_{j\neq i} \frac{1}{x_i - x_j} = 0.
\end{align*}
Moreover, for each $i=0,1,\cdots,q+K$, the difference $f(x_i) - P(x_i)$ has the same sign as $\nu(\{x_i\})$, for the signs of $\nu(\{x_i\})$ are also alternating. Therefore, 
\begin{align*}
\int f(x)\nu(dx) &= \int (f(x) - P(x))\nu(dx) + \int P(x)\nu(dx)\\
&= \int |f(x) - P(x)||\nu|(dx)  + 0\\
&= E_{q-1,K}(f;I)\cdot |\nu|(\bR) = E_{q-1,K}(f;I). 
\end{align*}
Now the proof is completed by considering the Jordan decomposition of $\nu = \nu_+ - \nu_-$ and choosing $\nu_0 = 2\nu_+, \nu_1 = 2\nu_-$. 

\subsection{Proof of Lemma \ref{lem.approx}}
We need some notions and results from approximation theory first. For functions defined on $[0,1]$, define the $r$-th order Ditzian--Totik modulus of smoothness by \cite{Ditzian--Totik1987}
\be
\omega_\varphi^r(f,t)_\infty \triangleq \sup_{0<h\le t} \|\Delta_{h\varphi(x)}^r f(x)\|_\infty, 
\ee
where $\varphi(x)\triangleq \sqrt{x(1-x)}$. This quantity is related to the polynomial approximation error via the following lemma. 
\begin{lemma}\cite{Ditzian--Totik1987}\label{lem.DT_modulus}
	For an integer $u>0$ and $n>u$, there exists some constant $M_u$ which depends on $u$ but not on $t\in (0,1)$ nor $f$, such that 
	\be
	E_{0,n}(f;[0,1]) &\le M_u\omega_\varphi^u(f,\frac{1}{n})_\infty,\\
	\frac{M_u}{n^u}\sum_{k=0}^n (k+1)^{u-1}E_{0,k}(f;[0,1]) &\ge \omega_\varphi^u(f,\frac{1}{n})_\infty.
	\ee
\end{lemma}

Take $u=1$, the second inequality together with the monotonicity of $E_{0,n}(f;[0,1])$ in $n$ yields
\be
E_{0,n}(f;[0,1]) &\ge \frac{1}{Dn}\sum_{k=n}^{Dn} E_{0,k}(f;[0,1])\\
&\ge \frac{\omega_\varphi^1(f,\frac{1}{Dn})_\infty}{M_1} - \frac{1}{Dn}\sum_{k=0}^{n-1} E_{0,k}(f;[0,1])\\ 
&\ge \frac{\omega_\varphi^1(f,\frac{1}{Dn})_\infty}{M_1} - \frac{E_{0,0}(f;[0,1])}{D}\nonumber, 
\ee
where we choose $D = \lceil 1/\sqrt{c} \rceil$, with the constant $c\in (0,1/2)$ is specified later. Since
\be
E_{q-1,n}(f_q;I) = \inf_{a_1,\cdots,a_{q-1}} E_{0,n}\left(x^{-q+\frac{r}{2}}+\sum_{k=1}^{q-1} \frac{a_k}{x^k}; I\right), 
\ee
it suffices to obtain a lower bound independent of $a_1,\cdots,a_{q-1}$ for the polynomial approximation error $E_{0,n}(x^{-q+r/2}+\sum_{k=1}^{q-1} a_kx^{-k}; [cn^{-2},1])$. Define $g(x)=x^{-q+r/2}+\sum_{k=1}^{q-1} a_kx^{-k}$, and let $\tilde{g}(x)=g(cn^{-2} + (1-cn^{-2})x)$ be the scaled version defined on $[0,1]$. We distinguish into two cases. 

First we consider the case where $E_{0,0}(\tilde{g};[0,1])\le C_1n^{2q-r}$ for some fixed constant $C_1=2c^{-q+r/2}$. By the definition of $\omega_\varphi^1(f,t)_\infty$ and the choice $D = \lceil  1/\sqrt{c} \rceil$, there exists universal constants $0<A<B$ (which are independent of everything) such that 
\be
\omega_\varphi^1(\tilde{g},\frac{1}{Dn})_\infty &\ge \sup_{t\in [A,B]} \left|g\left(\frac{t+1}{(Dn)^2}\right)-g\left(\frac{t}{(Dn)^2}\right)\right|\\
&= (Dn)^{2q-r}\sup_{t\in [A,B]} \left| h_{q-\frac{r}{2}}(t) + \sum_{k=1}^{q-1} (Dn)^{2k-2q+r}a_kh_k(t)\right|\\
&\ge (Dn)^{2q-r}\inf_{b_1,\cdots,b_{q-1}} \sup_{t\in [A,B]} \left|h_{q-\frac{r}{2}}(t) + \sum_{k=1}^{q-1} b_kh_k(t)\right|, 
\ee
where
\be
h_k(t) \triangleq (t+1)^{-k} - t^{-k}.
\ee
Since $r$ is not even, it is straightforward to verify that the functions $h_1,\cdots,h_{q-1},h_{q-\frac{r}{2}}$ are linearly independent in the interval $[A,B]$, we conclude that
\be
\omega_\varphi^1(\tilde{g},\frac{1}{Dn})_\infty \ge C_2(Dn)^{2q-r}, 
\ee
where the constant $C_2>0$ only depends on $r, q, A, B$ but not on $a_1,\cdots,a_{q-1}$ or $c,D$. Hence, in this case we have
\be\label{eq.case_I}
E_{0,n}(g;[cn^{-2},1]) \ge \frac{C_2(Dn)^{2q-r}}{M_1} - \frac{C_1n^{2q-r}}{D}, 
\ee
where neither of the constants depends on $a_1,\cdots,a_{q-1}$. 

Second we consider the case where $E_{0,0}(\tilde{g};[0,1])> C_1n^{2q-r}$. Since $2q-r>0$, we have
\be
E_{0,0}(\tilde{g};[0,1]) &\le \max_{x\in [cn^{-2},1]} \left|x^{-q+r/2}+\sum_{k=1}^{q-1} a_kx^{-k}\right| \\
&\le c^{-q+r/2}n^{2q-r} + \sum_{k=1}^{q-1} |a_k|c^{-k}n^{2k}\\
&\le \frac{1}{2}E_{0,0}(\tilde{g};[0,1])  + \sum_{k=1}^{q-1} |a_k|c^{-k}n^{2k}, 
\ee
and thus there exists some $j\in\{1,\cdots,q-1\}$ such that
\be
|a_j|c^{-j}n^{2j} \ge C_3E_{0,0}(\tilde{g};[0,1]),
\ee
where $C_3 = 1/(2q)$ is a numerical constant. Defining the interval $[A,B]$ as in the first case, we have
\be
\omega_\varphi^1(\tilde{g},\frac{1}{Dn})_\infty &\ge \sup_{t\in [A,B]} \left|g\left(\frac{t+1}{(Dn)^2}\right)-g\left(\frac{t}{(Dn)^2}\right)\right|\\
&= |a_j|(Dn)^{2j}\sup_{t\in [A,B]} \left| \frac{(Dn)^{2q-r-2j}}{a_j}h_{q-\frac{r}{2}}(t) + \sum_{k=1}^{q-1} (Dn)^{2k-2j}\frac{a_k}{a_j}h_k(t)\right|\\
&\ge |a_j|(Dn)^{2j}\inf_{b_1,\cdots,b_{q-1}} \sup_{t\in [A,B]} |h_j(t) + b_jh_{q-\frac{r}{2}}(t) + \sum_{k\neq j} b_kh_k(t)|\\
&\ge C_4|a_j|(Dn)^{2j}\\
&\ge C_3C_4\cdot E_{0,0}(\tilde{g};[0,1]), 
\ee
where the numerical constant $C_4>0$ (which only depends on $r,q,A,B$ but not on $a_1,\cdots,a_{q-1}$ or $c,D$) again follows from the linear independence of the functions $h_{q-\frac{r}{2}}, h_1,\cdots, h_{q-1}$ in $[A,B]$. Now in this case we have
\be
E_{0,n}(g;[cn^{-2},1]) &\ge \left(\frac{C_3C_4}{M_1} - \frac{1}{D}\right) \cdot E_{0,0}(\tilde{g};[0,1]) \\\label{eq.case_II}
&\ge C_1\left(\frac{C_3C_4}{M_1} - \frac{1}{D}\right)\cdot n^{2q-r}, 
\ee
where again none of the constants depends on $a_1,\cdots,a_{q-1}$. 

Now by (\ref{eq.case_I}) and (\ref{eq.case_II}), we plug in the definition $D=\lceil 1/\sqrt{c}\rceil$ and $C_1 = 2c^{-q+r/2}n^{2q-r}$, and use that the constants $C_2, C_3, C_4$ do not depend on $a_1,\cdots,a_{q-1}$ or $c,D$, to conclude that for $c>0$ small enough, we have
\be
E_{q,n}(x^{-q+r/2};[cn^{-2},1]) &= \inf_{a_1,\cdots,a_{q-1}} E_{0,n}(x^{-q+r/2}+\sum_{k=1}^{q-1} \frac{a_k}{x^k}; [cn^{-2},1])\\
&\ge c'n^{2q-r}, 
\ee
which is the desired result. \qed

\section{Acknowledgement}
We would like to thank Tsachy Weissman for the tremendous support and very helpful discussions. We are also grateful to an anonymous referee for detailed and helpful comments on improving this paper. 

	\bibliographystyle{alpha}
	\bibliography{biblio_adaptation}
\end{document}